\documentclass[letterpaper, 12pt]{article}
\usepackage{fix-cm}

\usepackage{graphicx}

\usepackage[T1]{fontenc}

\usepackage{xcolor}

\usepackage{tikz}
\usetikzlibrary{arrows.meta}
\usepackage[utf8]{inputenc}
\usepackage[english]{babel}
\usepackage{amsfonts, amsmath, amsthm, amssymb, bm}
\usepackage{graphicx}
\graphicspath{ {./images/} }
\usepackage{url}

\usepackage{geometry}
\usepackage{hyperref}
\hypersetup{
  hidelinks,
  urlcolor = black,
}

\newtheorem{theorem}{Theorem}[section]
\newtheorem{remark}{Remark}[section]
\newtheorem{definition}{Defination}[section]

\newtheorem{proposition}{Proposition}[section]

\usepackage[ruled,vlined]{algorithm2e}

\title{Application of an unbalanced optimal transport distance and a mixed L1/Wasserstein distance to full waveform inversion}
\author{Da Li, Michael P. Lamoureux, Wenyuan Liao}
\date{}

\begin{document}

\maketitle

\tableofcontents

\begin{abstract}

Full waveform inversion (FWI) is an important and popular technique in subsurface earth property estimation. However, using the least-squares norm in the misfit function often leads to the local minimum solution of the optimization problem, and this phenomenon can be explained with the cycle-skipping artifact. Several methods that apply optimal transport distances to mitigate the cycle-skipping artifact have been proposed recently. The optimal transport distance is designed to compare two probability measures. To overcome the mass equality limit, we introduce an unbalanced optimal transport (UOT) distance with Kullback-Leibler divergence to balance the mass difference. Also, a mixed L1/Wasserstein distance is constructed that can preserve the convex properties with respect to shift, dilation, and amplitude change operation. An entropy regularization approach and scaling algorithms are used to compute the distance and the gradient efficiently. Two strategies of normalization methods that transform the seismic signals into non-negative functions are provided. Numerical examples are provided to demonstrate the efficiency and effectiveness of the new method.

\end{abstract}

\section{Introduction}

Full waveform inversion (FWI) is a high resolution seismic imaging technique that is designed to recover subsurface structure and physical properties from seismic data.
It was proposed by Lailly \cite{lailly1983seismic} and Tarantola \cite{tarantola1984inversion} in the early 1980s and has been gaining popularity with the improvement of the computing capacity.
From mathematical point of view, FWI is a nonlinear PDE-constrained optimization problem with subsurface physical properties such as velocity and density as the control parameters, and the waveform received by the receivers as the state parameters \cite{fichtner2010full}.
Depending on different physical model, the constraint PDE can be wave equation, acoustic wave equation or elastic wave equation with proper boundary conditions or techniques to simulate wave propagating in an unbounded domain \cite{fichtner2010full}.
Because of the huge size of the problem, gradient based optimization methods such as gradient descent, nonlinear conjugate gradient method, L-BFGS and inexact Newton's method are usually implemented.
The gradient of misfit function generally can be calculated by the adjoint state method \cite{plessix2006review}.

In conventional methods, the L2 distance is used in the misfit function during optimization to minimize the difference between observed and synthetic data \cite{fichtner2010full}.
As a nonlinear optimization problem, FWI algorithm suffers the existence of local minima.
One of the reasons causing the local minima is the cycle-skipping artifact, which occurs as the phase difference between two seismic signals is larger then half of wavelength \cite{virieux2009overview}.
To mitigate this problem, the optimal transport (OT) distance or so called Wasserstein distance has been introduced to describe the difference between seismic signals \cite{engquist2013application} and later it has been applied to the seismic imaging \cite{engquist2018seismic, yong2019least} and FWI problem \cite{engquist2016optimal}.
Although it requires certain prerequisites such as mass equality and normalization, the OT distance has the convexity property with respect to shift, dilation, and amplitude change of signals \cite{engquist2016optimal} which is one of the main motivations of introducing OT distance to the FWI problem.
The convexity with respect to shift operation is demonstrated by Figure \ref{fig:ch1_1}.

\begin{figure}[h]
  \centering
  \includegraphics[width=1\textwidth]{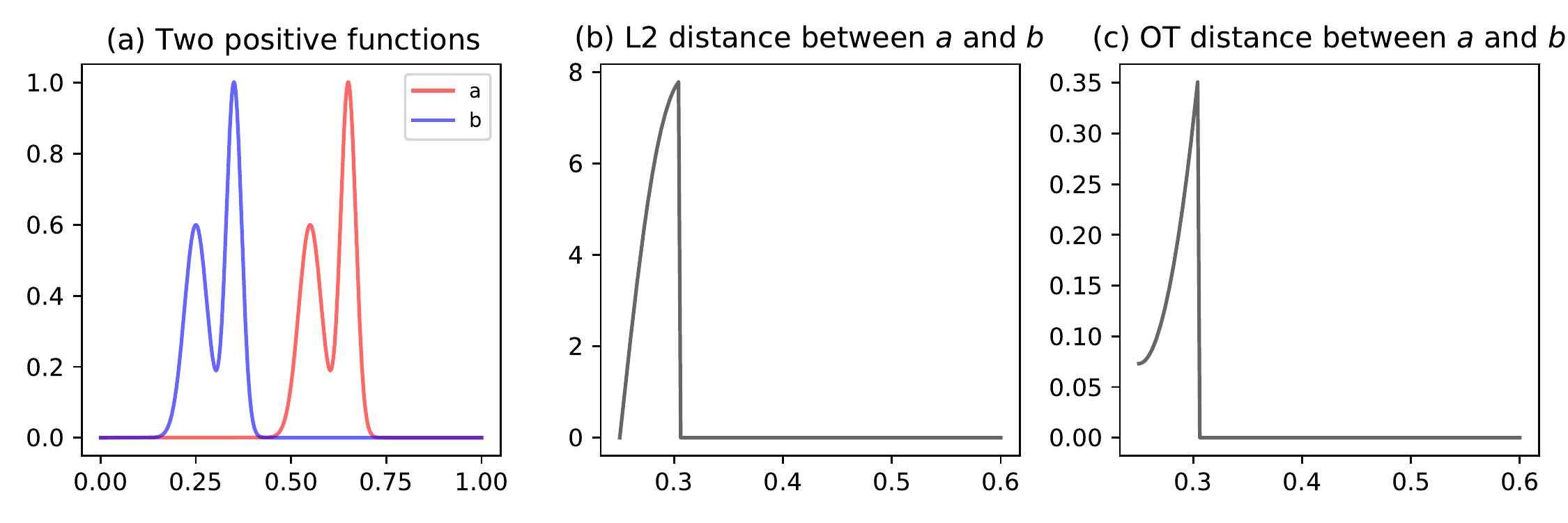}
  \caption{(a): Two positive signals $a$ and $b$. The position of $b$ is fixed. (b): L2 distance between $a$ and $b$ as $a$ is shifted from left to right. The cycle-skipping artifact exists at the local minimum near $0.35$. (c): Optimal transport distance between $a$ and $b$ as $a$ is shifted from left to right. Only a global minimum exists.}
  \label{fig:ch1_1}
\end{figure}

The optimal transport distance is designed to describe the difference between two positive finite measures with equal total mass.
Despite of the promising properties, the seismic signal is oscillating around $0$ and usually the mass equality condition is not satisfied.
Several works have been proposed to overcome those restrictions and integrate the OT distance to FWI problem.
In the first strategy, the non-negative and equal mass restrictions are overcame by connecting the 1-Wasserstein distance to KR norm \cite{bogachev2007measure} with the dual form of the Kantorovich problem.
And then the distance is computed by a proximal splitting strategy called simultaneous descent method of multipliers (SDMM) \cite{metivier2016measuring, metivier2016optimal}.
In \cite{yong2019misfit}, the 1-Wasserstein distance is evaluated through the dynamic formulation \cite{benamou2002monge}, and then solved by a primal-dual hybrid gradient method (PDHG) with line search.
Another strategy is to normalize the seismic signals into positive functions with equal mass first, then compute the OT distance.
In \cite{qiu2017mitigating, qiu2017full, yang2017analysis, yang2018application}, the seismic signals are normalized into positive functions with equal mass through normalization methods such as linear, quadratic and exponential functions.
Then the 2-Wasserstein distance between seismic signals can be evaluated either through a trace-by-trace technique or through numerically computation of Monge-Amp\`ere equation for 2D seismogram.
Furthermore, a graph space approach has been designed for comparing two signed signals with the optimal transport distance \cite{metivier2019graph}.

To overcome the equal mass limitation, the unbalanced optimal transport (UOT) problem is raised in \cite{benamou2003numerical} based on a dynamic approach.
Later several works have been proposed in both static and dynamic approaches \cite{piccoli2014generalized, chizat2018unbalanced, chizat2018scaling}. 
In this paper we introduce the UOT distance mainly based on the work in \cite{chizat2018scaling} to the FWI problem.
In addition to the UOT distance, a mixed L1/Wasserstein distance is constructed and then introduced to the FWI problem.
We prove that the proposed mixed distance inherits the convex properties with respect to shift and dilation of the OT distance.
Linear and exponential normalization methods have been studied to transform the signed signals into positive functions.
Numerical examples show that the UOT distance and mixed distance lead to a more convex behavior with respect to time shift of signal under proper normalization compared to the L2 distance and 1-Wasserstein distance in \cite{metivier2016optimal} and \cite{yong2019misfit}.
Compared to the 2-Wasserstein trace-by-trace strategy in \cite{yang2018application}, the UOT distance and mixed L1/Wasserstein distance provide more consistent adjoint source.
During the optimization procedure, the misfit function with UOT distance and mixed L1/Wasserstein distance provide more accurate update step compared to the L2 distance when a poor initial model is used.

The rest of this paper is organized as the follows.
In Section 2, we first give a general review of optimal transport problem. Then the UOT distance and mixed L1/Wasserstein distance are introduced.
Section 3 illustrates the normalization methods that transform the signals into positive functions.
Numerical examples are provided to demonstrate the convex behavior of the UOT distance and mixed L1/Wasserstein distance with the normalizations.
Section 4 provides the methodology of introducing UOT distance and mixed L1/Wasserstein distance to the FWI problem.
Three FWI numerical examples are shown in Section 5 to compare the inverse results with L2 distance, UOT distance, and mixed L1/Wasserstein distance.

\section{Optimal transport based distances}

In this section, we provide a short review of optimal transport problem and optimal transport distance in discrete form first.
Then, the optimal transport distance is evaluated through an entropy regularization approach and the Sinkhorn algorithm.
The unbalanced optimal transport distance used in this work is mainly based on the work \cite{chizat2018scaling}, and the similar entropy regularization approach is used for the numerical evaluation.
To maintain the convex properties with respect to shift and dilation operation, a mixed L1/Wasserstein distance is constructed in the last subsection.

The optimal transport problem is defined for probability measures.
Define the probability simplex as
\begin{align}
  \Sigma_n = \left\{ a\in\mathbb{R}_+^n \ \middle| \ \sum_{i=1}^n a_i = 1 \right\}.
\end{align}
Given two vectors $a, b \in \mathbb{R}^n$, the indicator function between $a$ and $b$ is defined as
\begin{align}
  \iota_{=}(a|b) = \begin{cases}
    0, \quad &a=b, \\ \infty, \quad &a\neq b.
  \end{cases}
\end{align}
Given a vector $a\in\mathbb{R}^n_+$, the entropy function is defined as:
\begin{align}
  E(a) = - \sum_{i=1}^n a_i (\log(a_i) - 1).
\end{align}
Given vector $a, b \in \mathbb{R}^n_+$, the Kullback-Leibler (KL) divergence is defined as:
\begin{align} \label{eq:KL_divergence}
  KL(a | b) = \sum_{i=1}^n a_i \log\left(\frac{a_i}{b_i}\right) - a_i + b_i.
\end{align}

\subsection{Optimal transport problem and entropy regularization}
The optimal transport problem has a long history and dates back to the 18th century \cite{monge1781memoire}.
The modern formulation is given by Kantorovich \cite{kantorovich2006translocation}.
Please refer to \cite{villani2008optimal,santambrogio2015optimal} for a comprehensive review.

We focus on the discrete measures in this work.
Denote the set of sampling points as $X = \{x_1, \cdots, x_n\} \in (\mathbb{R}^d)^n$ and $Y = \{y_1, \cdots, y_m\}  \in (\mathbb{R}^d)^m$, $d \in \mathbb{N}_+$.
Define two probability measures as:
\begin{align}\label{eq:two_discrete_measures}
  \alpha = \sum_{i=1}^n a_i \delta_{x_i}, \quad \beta = \sum_{i=1}^m b_i \delta_{y_i},
\end{align}
where $a\in \Sigma_n$, $b \in \Sigma_m$.
We use the matrix $P \in \mathbb{R}_{+}^{n \times m}$ to represent the transport plan in the discrete form.
The set of the transport plans in the discrete form is given by:
\begin{align}\label{eq:discrete_Kantorovich_feasible_set}
  \Pi_d(\alpha, \beta) = \left\{ P \in \mathbb{R}_{+}^{n \times m} \ \middle| \ P \bm{1}_m = a \ \text{and} \ P^{'} \bm{1}_n = b \right\},
\end{align}
where $\bm{1}_n$ is the row vector with $n$ entries and every entry is $1$.
Next, define the cost matrix $C \in \mathbb{R}^{n \times m}$ as
\begin{align}\label{eq:cost_matrix}
  C_{i,j} = c(x_i,y_j) = |x_i - y_j|^2,
\end{align}
where the distance function $c : X \times Y \to \mathbb{R}\cup \{+\infty\}$ is the square of Euclidean distance.

Given two discrete probability measures $\alpha$ and $\beta$ in $\mathbb{R}^d$ as \eqref{eq:two_discrete_measures}, and the cost matrix $C$ is defined by equation \eqref{eq:cost_matrix}.
The discrete Kantorovich problem is
\begin{align}\label{eq:discrete_Kantorovich}
  \min_{P \in \Pi_d(\alpha, \beta)} \left< P, C \right> = \sum_{i,j} C_{i,j} P_{i,j},
\end{align}
where the set of transport plan $\Pi_d(\alpha, \beta)$ is given by equation \eqref{eq:discrete_Kantorovich_feasible_set}.
Then, the square of $2$-Wasserstein distance between $\alpha$ and $\beta$ is defined as
\begin{align}\label{def:p_Wasserstein_distance}
  W_2^2(\alpha, \beta) = \min_{P \in \Pi_d(\alpha, \beta)} \left< P, C \right>.
\end{align}
When the set $X$ and $Y$ are fixed, the square of $2$-Wasserstein distance is with respect to vector $a$ and $b$ that can be written as $W_2^2(a,b)$.

The dual Kantorovich problem is given by
\begin{align}\label{eq:discrete_dual_problem}
  \max_{(\phi, \psi) \in R_C} \phi^{'} a + \psi^{'} b,
\end{align}
where the polyhedron $R_C$ of dual variables is
\begin{align}\label{eq:feasible_set_of_discrete_dual}
  R_C = \left\{ (\phi, \psi) \in \mathbb{R}^{n} \times \mathbb{R}^{m} \ \middle| \ \phi_i + \psi_j \leq C_{i,j}, \ 1 \leq i \leq n,\  1 \leq j \leq m \right\}.
\end{align}
The subdifferentiability of $W_2^2$ as a function of $a$ is given by the following proposition.
\begin{proposition}[\cite{cuturi2014fast}, Proposition 1]\label{prop:subdifferentiability}
  Given two discrete probability measure $\alpha$ and $\beta$ as equation \eqref{eq:two_discrete_measures}, and the cost matrix $C$ is defined by equation \eqref{eq:cost_matrix}.
  Any optimal dual variable $\phi$ of the dual problem \eqref{eq:discrete_dual_problem} is a subgradient of $W_2^2$ with respect to $a$.
\end{proposition}

The following convex properties with respect to shift and dilation operation are provided in \cite{yang2018optimal} and \cite{li2021novel}.

\begin{theorem} \label{thm:convex_wrt_shift}
  Suppose $\alpha$ and $\beta$ are two discrete probability measures defined by equation \eqref{eq:two_discrete_measures}, 
  let $P \in \Pi_d(\alpha, \beta)$ be the discrete optimal transport plan that rearranges $\alpha$ to $\beta$.
  The shift of discrete measure $\alpha$ with the direction $\eta \in \mathbb{R}^{d} $ and shift size $s>0$ is defined as
  \begin{align}
    \alpha_s = \sum_{i=1}^n a_i \delta_{x_i + s \eta}.
  \end{align}
  Then $W_2^2(\alpha_s, \beta)$ is convex with respect to the shift size $s$.
\end{theorem}

\begin{theorem} \label{thm:convex_wrt_dilation}
  Given $a \in \Sigma_n$ and a discrete set $X = \{x_1, \cdots, x_n\} \in (\mathbb{R}^d)^n$, two discrete probability measures are defined by
  \begin{align*}
    \alpha = \sum_{i=1}^n a_i \delta_{x_i}, \quad \alpha_A = \sum_{i=1}^n a_i \delta_{Ax_i}.
  \end{align*}
  Here $A$ is a dilation transform matrix which is symmetric positive definite, and $\alpha_A$ is the dilation of $\alpha$. 
  Let $P \in \Pi_d(\alpha_A, \alpha)$ be the discrete transport plan that rearranges $\alpha_A$ to $\alpha$.
  Then, $P$ is the optimal transport plan with $P = \text{diag} (a)$, and $W_2^2(\alpha_A, \alpha)$ is convex with respect to the eigenvalues of $A$.
\end{theorem}

The computational cost of solving the above discrete Kantorovich problem through network simplex or interior point methods for two $n$-dimensional vectors is at least $O(n^3\log(n))$ \cite{cuturi2013sinkhorn, pele2009fast}, that prevents the widespread use of the Wasserstein distance for large-scale problem.
Other strategies has been came up with, such as through a dynamic form formulation \cite{benamou2000computational}, through the connection between OT problem and Monge-Amp\`ere equation \cite{loeper2005numerical, froese2011convergent}.
Besides the above methods, the entropy regularization method is the most popular numerical method for the optimal transport problem.
We provide a short review of using the Sinkhorn algorithm to solve the entropy regularized optimal transport problem based on the work \cite{cuturi2013sinkhorn, cuturi2014fast}.

Given discrete probability measures $\alpha$ and $\beta$ defined by equation \eqref{eq:two_discrete_measures}, consider the entropy regularized Kantorovich problem:
\begin{align}\label{eq:regularized_KP}
  P_\varepsilon = \arg \min_{P \in \Pi_{d}(\alpha, \beta)} \left< P, C \right> - \varepsilon E(P),
\end{align}
where $\varepsilon > 0$ is the regularization parameter, $\Pi_{d}(\alpha, \beta)$ is the set of optimal transport plan defined by \eqref{eq:discrete_Kantorovich_feasible_set}.
It is straightforward to see that
\begin{equation}
  \begin{aligned}
    P_\varepsilon = \arg \min_{P\in\Pi_{d}(\alpha,\beta)} KL(P|K),
  \end{aligned}
\end{equation}
where $K_{i,j} = e^{-C_{i,j}/\varepsilon}$.
With the above $P_\varepsilon$, the square of the regularized 2-Wasserstein distance is defined as:
\begin{align}
  W_{2,\varepsilon}^2 (\alpha, \beta) &= \left<P_\varepsilon, C \right>.
\end{align}
This is also denoted as the Sinkhorn distance in the work \cite{cuturi2013sinkhorn}.

Consider the Lagrangian of the problem \eqref{eq:regularized_KP},
\begin{align}
  \mathcal{L} (P_\varepsilon, \phi, \psi) = \sum_{i,j} (P_\varepsilon)_{i,j} C_{i,j} + \varepsilon (P_\varepsilon)_{i,j} \left(\log (P_\varepsilon)_{i,j} - 1\right) + \phi' \left(a - P_\varepsilon \bm{1}_m \right) + \psi' \left(b - P_\varepsilon^{'} \bm{1}_n \right).
\end{align}
The first order optimality condition provides that
\begin{align}
  \frac{\partial \mathcal{L} (P_\varepsilon, \phi, \psi)}{\partial (P_\varepsilon)_{i,j}} = C_{i,j} + \varepsilon \log (P_\varepsilon)_{i,j} - \phi_i - \psi_j = 0,
\end{align}
which equivalents to
\begin{align}
  P_\varepsilon = \text{diag}(u) K \text{diag} (v),
\end{align}
where $u_i = e^{\phi_i / \varepsilon}$, $v_j = e^{\psi_i / \varepsilon}$.
Then, the dual problem can be derived with the Lagrangian as:
\begin{align} \label{eq:regularized_KP_dual}
  \max_{\phi, \psi \in \mathbb{R}^{n} \times \mathbb{R}^m} \phi^{'} a + \psi^{'} b - \varepsilon\sum_{i,j} e^{-(C_{i,j} - \phi_i - \psi_j)/\varepsilon}.
\end{align}


The above dual problem can be solved with a matrix scaling algorithm named Sinkhorn algorithm.
For more discussion please refer to \cite{cuturi2013sinkhorn,cuturi2014fast}.
The Sinkhorn algorithm is given by Algorithm \ref{algorithm:sinkhorn}.
\begin{algorithm}[h] \label{algorithm:sinkhorn}
  \SetAlgoLined
  Input: $\alpha$, $\beta$, $C$, $\varepsilon$. \\
  Initialization: matrix $K$ with $K_{i,j} = e^{-C_{i,j}/\varepsilon}$, $u = \bm{1}_n$, $v = \bm{1}_m$.\\
  \While{not converged}{
    Update vector $u$ with $u_i = a_i / (Kv)_i$.\\
    Update vector $v$ with $v_j = b_j / (K^{'}u)_j$.
  }
  Compute the transport plan matrix $P_\varepsilon = \text{diag}(u) K \text{diag} (v)$. \\
  \Return{
    The Sinkhorn distance $W_{2,\varepsilon}^2 (\alpha,\beta)  = \left< P_\varepsilon, C\right>$. \\
    The gradient of $W_{2,\varepsilon}^2 (\alpha,\beta)$ with respect to $a$ is $(\nabla_a W_{2,\varepsilon}^2 (\alpha,\beta) )_i = \varepsilon \log(u_i) - \varepsilon/n \sum_{j} \log(u_j)$.  
  }
  \caption{Sinkhorn algorithm.}
\end{algorithm}
The stopping criteria can be designed in a variety of ways, for example, the iteration process can be terminated after a certain number of iterations.
Another strategy is: denote the discrete transport plan at the $k$-th iteration as $P_\varepsilon^k = \text{diag}(u^k) K \text{diag} (v^k)$, the distance of the $k$-th iteration as $d_1 = \left< P_\varepsilon^k, C\right>$, and the distance of the previous iteration as $d_0 = \left< P_\varepsilon^{k-1}, C\right>$.
The iteration can be terminated when $|d_1/d_0 - 1| \leq \eta$, here $\eta$ is a threshold variable to control the accuracy.

\subsection{Unbalanced optimal transport distance}
In this subsection, we introduce the UOT distance mainly based on the work in \cite{chizat2018scaling}.

Define two positive discrete measures:
\begin{align}\label{eq:two_discrete_measure1}
  \alpha = \sum_{i=1}^n a_i \delta_{x_i}, \quad \beta = \sum_{i=1}^m b_i \delta_{y_i},
\end{align}
where $a = (a_1, \cdots, a_n) \in \mathbb{R}_+^{n}$ and $b = (b_1, \cdots, b_m) \in \mathbb{R}_+^{m}$.

The unbalanced optimal transport problem is a generalization of the optimal transport problem.
To relax the marginal constraints in the discrete Kantorovich problem, the unbalanced optimal transport problem is defined as:
\begin{align}
    \min_{P\in\mathbb{R}^{n\times m}} \left<P,C \right> + F_a(P\bm{1}_m) + F_b(P^{'} \bm{1}_n), \label{eq:UOT_problem}
\end{align}
here both $F_a$ and $F_b$ are proper convex functions.
For example, when $F_a$ and $F_b$ are the indicator function:
\begin{align}
  F_a(P\bm{1}_m) = \iota_{\{=\}}(P\bm{1}_m|a), \quad F_b(P^{'}\bm{1}_n) = \iota_{\{=\}}(P^{'}\bm{1}_n|b),
\end{align}
it can be easily checked that the unbalanced optimal transport problem \eqref{eq:UOT_problem} coincides with the discrete Kantorovich problem \eqref{eq:discrete_Kantorovich} when $\sum_i a_i = \sum_i b_i = 1$.
In this work we consider the case when
\begin{align}
  F_a(P\bm{1}_m) = \varepsilon_u KL(P\bm{1}_m|a), \quad F_b(P^{'}\bm{1}_n) = \varepsilon_u KL(P^{'}\bm{1}_n|b).
\end{align}
Here $F_a$ and $F_b$ are the Kullback-Leibler divergence between vectors given in equation \eqref{eq:KL_divergence} which measures the differences between $P\bm{1}_m$ and $a$, $P^{'}\bm{1}_n$ and $b$.
And the parameter $\varepsilon_u$ controls the weight of the mass balancing term in \eqref{eq:UOT_problem}.

Similar to the Wasserstein distance \eqref{def:p_Wasserstein_distance}, the unbalanced optimal transport distance with the square Euclidean ground cost between vector is:
\begin{align} \label{UOT_distance}
  W_{2,\varepsilon_u}^2(\alpha,\beta) = \min_{P\in\mathbb{R}^{n\times m}} \left<P,C \right> + \varepsilon_u KL(P\bm{1}_m|a) + \varepsilon_u KL(P^{'}\bm{1}_n|b),
\end{align}
where the cost matrix $C$ is defined as $C_{i,j} = |x_i-y_j|^2$.
When the set $X$ and $Y$ are fixed, the above UOT distance is with respect to vector $a$ and $b$ that can be written as $W_{2,\varepsilon_u}^2(a,b)$.

As we discussed in the previous subsection, the entropy regularization method is a proper choice to evaluate the Wasserstein distance and the gradient.
Given the regularization parameter $\varepsilon > 0$, consider the entropy regularized UOT problem:
\begin{align}
  \min_{P \in\mathbb{R}^{n\times m}} \left<P,C \right> -\varepsilon E(P) + \varepsilon_u KL(P\bm{1}_m|a) + \varepsilon_u KL(P^{'}\bm{1}_n|b). \label{eq:Regularized_UOT}
\end{align}
Same as we discussed in the previous subsection, the above equation can be rewritten as
\begin{align}
  \min_{P\in\mathbb{R}^{n\times m}} \varepsilon KL(P|K) + \varepsilon_u KL(P\bm{1}_m|a) + \varepsilon_u KL(P^{'}\bm{1}_n|b),
\end{align}
where $K$ is defined as $K_{i,j} = e^{-C_{i,j}/ \varepsilon}$.

\begin{definition} \label{def:UOT}
  Given positive discrete measures $\alpha, \beta$ with equation \eqref{eq:two_discrete_measure1}.
  Define the ground cost matrix $C$ by $C_{i,j} = |x_i-y_j|^2$.
  With the regularization parameter $\varepsilon$ and the mass balancing parameter $\varepsilon_u$, the regularized unbalanced optimal transport distance between $\alpha$ and $\beta$ is
  \begin{equation}
    \begin{aligned}
      W_{2,\varepsilon_u, \varepsilon}^2 (\alpha,\beta) =& \left<P_{\varepsilon}, C \right> + \varepsilon_u KL(P_{\varepsilon}\bm{1}_m|a) + \varepsilon_u KL(P^{'}_{\varepsilon}\bm{1}_n|b), \\
      \text{where } P_{\varepsilon } =& \arg\min_{P\in\mathbb{R}^{n\times m}} \varepsilon KL(P|K) + \varepsilon_u KL(P\bm{1}_m|a) + \varepsilon_u KL(P^{'}\bm{1}_n|b). \label{eq:Primal_UOT}  
    \end{aligned}
  \end{equation}
  Here $KL(\cdot|\cdot)$ is the Kullback-Leibler divergence between two matrices or vectors.
  And $K_{i,j} = e^{-C_{i,j}/ \varepsilon}$.
\end{definition}

Equation \eqref{eq:Primal_UOT} is denoted as the primal problem. 
The dual problem is needed to compute the unbalanced optimal transport distance.
\begin{theorem} [\cite{chizat2018scaling}, Theorem 3.2] \label{thm:Dual_thm}
  The dual problem of \eqref{eq:Primal_UOT} is
  \begin{align}
    \max_{\phi,\psi\in\mathbb{R}^n_+} \sum_{i,j} -\varepsilon_u a_i \left(e^{-\phi_i/\varepsilon_u}-1\right) - \varepsilon_u b_j \left(e^{-\psi_j/\varepsilon_u}-1\right)
    - \varepsilon K_{i,j} \left(e^{\phi_i/\varepsilon} e^{\psi_j/\varepsilon}-1\right), \label{eq:Dual_UOT}
  \end{align}
  where the matrix $K$ is defined by $K_{i,j} = e^{-C_{i,j}/\varepsilon}$.
  Strong duality holds for the primal and the dual problem.
  The minimization is attained for a unique $P_{\varepsilon }^*$ for the primal problem \eqref{eq:Primal_UOT}.
  And $\phi^*$, $\psi^*$ maximize the dual problem \eqref{eq:Dual_UOT} if and only if:
  \begin{align}
    (P_{\varepsilon }^*)_{i,j} = e^{\phi_i^*/\varepsilon} K_{i,j} e^{\psi_j^*/\varepsilon}.
  \end{align}
\end{theorem}
\begin{proposition} \label{scaling_prop}
  Given matrix $K$, coefficient $\varepsilon$ and $\varepsilon_u$ in consistent with Theorem \ref{thm:Dual_thm}.
  Suppose $\phi^*$, $\psi^*$ solve the dual problem \eqref{eq:Dual_UOT}, let $(u^*, v^*) \in\mathbb{R}^n_+ \times \mathbb{R}_+^m$ with $u_i^* = e^{\phi_i^*/\varepsilon}$ and $v_j^* = e^{\psi_j^*/\varepsilon}$.
  Then,
  \begin{align}
    u^{*}_i = \left(\frac{a_i}{\sum_j K_{i,j}v_j^{*}}\right)^{\frac{\varepsilon_u}{\varepsilon_u+\varepsilon}},\quad
  v^{*}_j = \left(\frac{b_j}{\sum_i K_{i,j}u_i^{*}}\right)^{\frac{\varepsilon_u}{\varepsilon_u+\varepsilon}}.
  \end{align}
\end{proposition}
The above proposition can be easily checked by computing the first order optimality condition of the dual problem \eqref{eq:Dual_UOT}.
The following remark provides an algorithm to compute the UOT distance with the entropy regularization as Definition \ref{def:UOT}.
\begin{remark}
  Starting with an initial value $v^{(0)} = \bm{1}_m$, the dual problem can be computed through a coordinate ascent algorithm:
  for the $k$-th iteration,
  \begin{align}
    u^{(k+1)}_i = \left(\frac{a_i}{\sum_j K_{i,j}v_j^{(k)}}\right)^{\frac{\varepsilon_u}{\varepsilon_u+\varepsilon}},\quad
  v^{(k+1)}_j = \left(\frac{b_j}{\sum_i K_{i,j}u_i^{(k+1)}}\right)^{\frac{\varepsilon_u}{\varepsilon_u+\varepsilon}}.
  \end{align}
  Suppose the coordinate ascent algorithm converges with $u^*, v^*$, the transport plan matrix $P_{\varepsilon }^*$ in \eqref{eq:Primal_UOT} can be computed as
  \begin{align}
    (P_{\varepsilon }^*)_{i,j} = u^*_i K_{i,j} v^*_j.
  \end{align}
\end{remark}
Also, the gradient of regularized UOT distance can be achieved through the following remark.
\begin{remark} \label{prop_adj_source}
  Suppose $P^*$, $\phi^*$ and $\psi^*$ solve the primal problem \eqref{eq:Primal_UOT} and dual problem \eqref{eq:Dual_UOT}, the gradient of regularized unbalanced optimal transport distance with respect to $a$ is:
  \begin{align}
    \nabla_{a} W_{2,\varepsilon_u, \varepsilon}^2 (\alpha,\beta) = -\varepsilon_u \left(e^{-\phi^*/\varepsilon_u} - 1\right).
  \end{align}
\end{remark}

The algorithm to compute the regularized UOT distance and the gradient is given by Algorithm \ref{algorithm:UOT}.
Similar stopping criteria of Algorithm \ref{algorithm:sinkhorn} can be designed.
For more information about the above remarks and Algorithm \ref{algorithm:UOT}, please refer to the work \cite{chizat2018scaling}.
\begin{algorithm}[h] \label{algorithm:UOT}
\SetAlgoLined
  Input: $C$, $\varepsilon_u$, $\varepsilon$\\
  Initialization: $K_{i,j}=e^{-C_{i,j}/\varepsilon}$, $v = \bm{1}_m$, \\
  \While{not converged}{
    Update vector $u$ with $u_i = a_i / (Kv)_i^{\varepsilon_u / (\varepsilon_u+\varepsilon)} $.\\
    Update vector $v$ with $v_j = b_j / (K^{'}u)_j^{\varepsilon_u / (\varepsilon_u+\varepsilon)}$.
  }
  Compute transport matrix $P_{\varepsilon }$ with $(P_{\varepsilon })_{i,j} = u_i K_{i,j} v_j$ \\
  \Return{
    The regularized UOT distance: 
    \begin{align*}
      W_{2,\varepsilon_u, \varepsilon}^2 (\alpha,\beta) ={}& \sum_{i,j}  (P_{\varepsilon })_{i,j}C_{i,j} + \varepsilon_u \left((P_{\varepsilon } \bm{1}_m)_i\log\left(\frac{( P_{\varepsilon } \bm{1}_m)_i}{a_i}\right)-(P_{\varepsilon } \bm{1}_m)_i+a_i\right)  \\
      &+ \varepsilon_u \left((P_{\varepsilon } ^{'}\bm{1}_n)_j\log\left(\frac{(P_{\varepsilon } ^{'}\bm{1}_n)_j}{b_j}\right)-(P_{\varepsilon } ^{'}\bm{1}_n)_j+b_j\right).
    \end{align*}\\
    The gradient of $W_{2,\varepsilon_u,\varepsilon}^2 (\alpha,\beta)$ with respect to $a$ is $(\nabla_a W_{2,\varepsilon_u,\varepsilon}^2 (\alpha,\beta) )_i = -\varepsilon_u \left(u_i^{-\varepsilon/\varepsilon_u}-1\right)$.
  }
  \caption{Scaling algorithm for regularized UOT distance and gradient}
\end{algorithm}

\subsection{Mixed L1/Wasserstein distance}


A mixed L1/Wasserstein distance is constructed in this subsection.
The concept of the mixed distance with Wasserstein distance is not new.
In the work of \cite{benamou2001mixed}, a mixed L2/Wasserstein is provided through the dynamic form of optimal transport problem.

Start with the positive discrete measures defined as equation \eqref{eq:two_discrete_measure1}.
Denote the normalized $\alpha$ and $\beta$ as
\begin{align}\label{eq:normalized_measures}
  \hat \alpha = \frac{1}{\|a\|_1} \sum_{i=1}^n a_i \delta_{x_i} = \sum_{i=1}^n \hat a_i \delta_{x_i}, \quad \hat \beta = \frac{1}{\|b\|_1} \sum_{i=1}^m b_i \delta_{y_i} = \sum_{i=1}^m \hat b_i \delta_{y_i}.
\end{align}
It is straight forward to see that $\hat a = (\hat a_1, \cdots, \hat a_n) \in \Sigma_n$, $\hat b = (\hat b_1, \cdots, \hat b_m)\in \Sigma_m$.
The $2$-Wasserstein distance between $\hat\alpha$ and $\hat\beta$ is well defined.

The following definition provides the mixed L1/Wasserstein distance.
\begin{definition}[Mixed L1/Wasserstein distance]
  Given two discrete measures $\alpha$ and $\beta$ as equation \eqref{eq:two_discrete_measure1}, the cost matrix $C$ is defined by $C_{i,j} = |x_i - y_j|^2$.
  The mixed L1/Wasserstein distance between $\alpha$ and $\beta$ is
  \begin{align}\label{def:mixed_distance}
    \bar W_2(\alpha, \beta) = W_2(\hat \alpha, \hat \beta) + \left| \|a\|_1 - \|b\|_1 \right|,
  \end{align}
  where $\hat \alpha$ and $\hat \beta$ are the normalized measures given by equation \eqref{eq:normalized_measures}.
\end{definition}
In the above definition, the first term is the $2$-Wasserstein distance between normalized measures $\hat \alpha$ and $\hat \beta$, which describes the shape difference between $\alpha$ and $\beta$.
The second term is the L1 term, which is the absolute value of the mass difference between $\alpha$ and $\beta$.
When the set $X$ and $Y$ are fixed, the mixed L1/Wasserstein distance is with respect to vector $a$ and $b$ that can be written as $\bar W_2(a,b)$.
The following proposition describes the metric property of the above mixed distance.
\begin{proposition}
  The mixed L1/Wasserstein distance defined in equation \eqref{def:mixed_distance} is a distance over the set of positive discrete measures.
\end{proposition}
\begin{proof}
  It is clear that $\bar W_2(\alpha, \beta)$ is nonnegative.
  Suppose $\alpha = \beta$, then $\|a\|_1 = \|b\|_1$, and $\left| \|a\|_1 - \|b\|_1 \right| = 0$.
  Also $W_2(\hat \alpha, \hat \beta) = 0$ by the metric property of $W_2$, we have $\bar W_2(\alpha, \beta) = 0$.
  On the other hand, if $\bar W_2(\alpha, \beta) = 0$ then $W_2(\alpha, \beta) = 0$ and $\left| \|a\|_1 - \|b\|_1 \right| = 0$, which leads to $\alpha = \beta$.

  The symmetric property of $\bar W_2(\alpha, \beta)$ comes directly form the metric property of $W_2(\alpha, \beta)$.

  Denote $\gamma = \sum_{i = 1}^k c_i \delta_{z_i}$, where $c \in \mathbb{R}_+^{k}$ and the set of sampling point $Z = \{z_1, \cdots, z_k\}$.
  \begin{equation}
    \begin{aligned}
      \bar W_2(\alpha, \gamma)  &= W_2(\hat \alpha, \hat \gamma) + \left| \|a\|_1 - \|c\|_1 \right| \\
      &= W_2(\hat \alpha, \hat \gamma) + \left| \|a\|_1 - \|b\|_1 + \|b\|_1 - \|c\|_1 \right| \\
      &\leq  W_2(\hat \alpha, \hat \beta) + W_2(\hat \beta, \hat \gamma) + \left| \|a\|_1 - \|b\|_1 \right| + \left| \|b\|_1 - \|c\|_1 \right| \\
      &= \bar W_2(\alpha, \beta) + \bar W_2(\beta, \gamma).  
    \end{aligned}
  \end{equation}
  The inequality follows by the metric property of $2$-Wasserstein distance and the triangle inequality.
\end{proof}

For the variational problem when a discrete measure $\beta$ is given, we are going to apply the optimization algorithm to find a measure $\alpha$ which is close to $\beta$.
Since
\begin{equation}
  \begin{aligned}
    \bar W_2^2(\alpha, \beta) = \left( W_2(\hat \alpha, \hat \beta) + \left| \|a\|_1 - \|b\|_1 \right| \right)^2 \leq 2\left( W_2^2(\hat \alpha, \hat \beta) + \left( \|a\|_1 - \|b\|_1 \right)^2 \right),
  \end{aligned}
\end{equation}
define
\begin{align}
  J(\alpha; \beta) = W_2^2(\hat \alpha, \hat \beta) + \left( \|a\|_1 - \|b\|_1 \right)^2 . \label{eq:variational_fn_mixed}
\end{align}
The above inequality follows by the triangle inequality.
In this case, to minimize the mixed L1/Wasserstein distance between $\alpha$ and $\beta$, it is sufficient to minimize the objective function $J(\alpha; \beta)$ with respect to $\alpha$.

By the construction of the discrete measure $\alpha$, the objective function is controlled by vector $a$ when the sampling set $X$ is fixed, i.e. 
\begin{align}
  J(\alpha; \beta) = J(a). \label{eq:variational_fn_mixed_aX}
\end{align}
The following proposition provides the subdifferentiability of the objective function with respect to $a$.
\begin{proposition}
  The objective function $J(a)$ defined by equation \eqref{eq:variational_fn_mixed} and \eqref{eq:variational_fn_mixed_aX} is subdifferentiable with respect to $a$.
\end{proposition}
\begin{proof}
  The subdifferentiability of $J(a) = J(\alpha; \beta)$ with respect to vector $a$ straightly follows by Proposition \ref{prop:subdifferentiability}.
\end{proof}

The following propositions show that the objective function \eqref{eq:variational_fn_mixed_aX} with the proposed mixed distance retains the convex properties of the square of $2$-Wasserstein distance.
\begin{proposition}
  Given two discrete measures $\alpha$ and $\beta$.
  Given the shift direction $\eta \in \mathbb{R}^d$ and the shift size $s>0$, the shift of discrete measure $\alpha$ is denoted as
  \begin{align}
    \alpha_s = \sum_{i=1}^n a_i \delta_{x_i + s\eta},
  \end{align}
  The objective function $J(s) = J(\alpha_s; \beta)$ is defined by equation \eqref{eq:variational_fn_mixed}.
  Then $J(s)$ is convex with respect to $s$.
\end{proposition}
\begin{proof}
  By the construction of the objective function
  \begin{align}
    J(s) = J(\alpha_s, \beta) = W_2^2(\hat \alpha_s, \hat \beta) + \left( \|a\|_1 - \|b\|_1 \right)^2.
  \end{align}
  The convexity of $J(s)$ follows the $W_2^2(\hat \alpha_s, \hat \beta)$ is convex with respect to $s$ by Theorem \ref{thm:convex_wrt_shift}.
\end{proof}
\begin{proposition}
  Given a discrete measure $\alpha = \sum_{i=1}^n a_i \delta_{x_i}$, where $a \in \mathbb{R}_+^{n}$.
  The dilation of measure $\alpha$ is given by
  \begin{align}
    \alpha_A = \sum_{i=1}^n a_i \delta_{A x_i},
  \end{align}
  here $A$ is a dilation transform matrix which is symmetric positive definite.
  Let $\lambda_1, \cdots, \lambda_d$ be the eigenvalues of matrix $A$.
  The objective function $J(\lambda_1, \cdots, \lambda_d) = J(\alpha_A; \alpha)$ is defined by equation \eqref{eq:variational_fn_mixed}.
  Then $J(\lambda_1, \cdots, \lambda_d)$ is convex with respect to $\lambda_1, \cdots, \lambda_d$.
\end{proposition}
\begin{proof}
  By the construction of the objective function
  \begin{align}
    J(\lambda_1, \cdots, \lambda_d) = J(\alpha_A; \alpha) = W_2^2(\hat \alpha_A, \hat \alpha) + \left( \|a\|_1 - \|a\|_1 \right)^2.
  \end{align}
  The convexity of $J(\lambda_1, \cdots, \lambda_d)$ follows that $W_2^2(\hat \alpha_A, \hat \alpha)$ is convex with respect to $\lambda_1, \cdots, \lambda_d$ by Theorem \ref{thm:convex_wrt_dilation}.
\end{proof}
\begin{proposition}
  Given two discrete measures $\alpha$ and $\beta$.
  Denote the mass change of measure $\alpha$ by
  \begin{align}
    \alpha_k = \sum_{i=1}^n (k a_i) \delta_{x_i},
  \end{align}
  where $k \in \mathbb{R}_+$.
  The objective function $J(k) = J(\alpha_k; \beta)$ is defined by equation \eqref{eq:variational_fn_mixed}.
  Then $J(k)$ is convex with respect to $k$.
\end{proposition}
\begin{proof}
  By the construction of the objective function
  \begin{align}
    J(k) = J(\alpha_k; \beta) = W_2^2(\hat \alpha_k, \hat \beta) + \left( k\|a\|_1 - \|b\|_1 \right)^2.
  \end{align}
  Since $\hat \alpha_k = \hat \alpha$, then the convexity of $J(k)$ follows by the second term of the above equation.
\end{proof}

The numerical evaluation of $J(\alpha; \beta)$ is straightforward.
Suppose the support of $\alpha$ and $\beta$ are fixed, the entropy regularization method can be used to define a regularized $J_\varepsilon$ with
\begin{align}
  J_\varepsilon(\alpha; \beta) = W_{2, \varepsilon}^2 (\hat a, \hat b) + (\|a\|_1 - \|b\|_1)^2.
\end{align}
The first part of the $J_\varepsilon(\alpha; \beta)$ can be evaluated with the Sinkhorn algorithm \ref{algorithm:sinkhorn}.
When the support of $\alpha$ is fixed, the regularized objective function $J_\varepsilon(\alpha; \beta)$ can be written as $J_\varepsilon (a)$, and the gradient of $J_\varepsilon(a)$ with respect to $a$ is given by
\begin{align}
  \nabla J_\varepsilon(a) = (D \hat a)' \nabla_{\hat a} W_{2, \varepsilon}^2 (\hat a, \hat b) + 2 (\|a\|_1-\|b\|_1) \bm{1}_n,
\end{align}
where $\nabla_{\hat a} W_{2, \varepsilon}^2 (\hat a, \hat b)$ is the gradient of $W_{2, \varepsilon}^2 (\hat a, \hat b)$ with respect to the first entry.
The Jacobian matrix $D \hat a$ is given by
\begin{align}
  (D \hat a)_{i,j} = \frac{\partial \hat a_i}{\partial a_j} = \begin{cases}
    -\frac{a_i}{(\sum_k a_k)^2}, \quad &\text{if } i \neq j,\\
    \frac{1}{\sum_k a_k} - \frac{a_i}{(\sum_k a_k)^2}, \quad &\text{if } i = j.
  \end{cases}
\end{align}
In practice, a mass balancing parameter $\lambda_m>0$ can be introduced to the $J_\varepsilon$ as
\begin{align}
  J_\varepsilon(\alpha; \beta) = W_{2,\varepsilon}^2(\hat a, \hat b) + \lambda_m(\|a\|_1 - \|b\|_1)^2.
\end{align}

\section{Normalization methods}

As discussed in the previous section, the OT distance is convex with respect to shift and dilation, and that is our initial idea to introduce the OT distance to the seismic inverse problem.
However, normalization methods are needed to extend the OT distance to signed signals, and the convex properties may be affected.
In this section, we discuss the convex behavior with both linear and exponential normalization for both UOT distance and mixed L1/Wasserstein distance.

Instead of focusing on the theoretical properties, we are working with numerical examples to show how the normalizations behaves.
The normalization methods here can only partially solve the problem that generalizing the OT distance to compare the difference between signed signals.
However, the numerical examples provided in this section and in the following sections show that by introducing the optimal transport based distances, the inverse results of the FWI problem are indeed improved in certain cases.

Since the trace-by-trace strategy is going to be used, we focus on comparing the difference between one-dimensional signals with UOT distance and the mixed L1/Wasserstein distance.
The signal $a(t)$ and $b(t)$ are defined on the time domain.
When $t = (\delta_{t_1}, \cdots, \delta_{t_n})$ is fixed, the signals $a(t)$ and $b(t)$ can be represented as $n$-dimensional vectors $a = (a_1, \cdots, a_n)\in \mathbb{R}^n$, $b= (b_1, \cdots, b_n)\in \mathbb{R}^n$.

The L2 distance, UOT distance, and mixed L1/Wasserstein distance are going to be compared in several numerical experiments, and we use $d(\cdot, \cdot)$ to represent the distance used in the objective functions.
Let $d_u(\cdot, \cdot)$ represent the UOT distance 
\begin{align} \label{eq:UOT_distance}
  d_u(a, b) &= W_{2,\varepsilon_u, \varepsilon}^2(a,b),
\end{align}
and $d_m(\cdot, \cdot)$ represent the mixed L1/Wasserstein distance,
\begin{align} \label{eq:mixed_distance}
  d_m(a,b) &= J(a) = W_{2,\varepsilon}^2 (\hat a, \hat b) + \lambda_m (\|a\|_1 - \|b\|_1)^2,
\end{align}
where $J(a)$ is defined by equation \eqref{eq:variational_fn_mixed}.
Both UOT distance and mixed L1/Wasserstein distance can be evaluated through the entropy regularization approach.
The smaller the regularization coefficient $\varepsilon$ is used, the more accurate result we can achieve.
On the other hand, the regularization coefficient $\varepsilon$ can not be too small due to the machine precision.
In this work, the regularization coefficients are chosen as small as possible in the numerical experiments.

We focus on the linear and exponential normalizations in this work.
Given a normalization parameter $k$, the linear normalization is defined as
\begin{align} \label{eq:linear_normalization}
  h_{\text{l}} (a,k)(t) = a(t) + k,
\end{align}
and the exponential normalization is defined as
\begin{align} \label{eq:exp_normalization}
  h_{\text{e}} (a,k)(t) = e^{k a(t)}.
\end{align}
We demonstrate the behavior of the above two normalizations with numerical examples in the following part of this section.



Usually, the seismic event can be approximated with a linear combination of the Ricker wavelet, i.e.,
\begin{align}
  s(t) = \left(1 - \frac{(t-t_0)^2}{\sigma^2}\right) e^{\frac{-(t-t_0)^2}{2\sigma^2}}.
\end{align}

\begin{figure}[h]
  \centering
  \includegraphics[width=1\textwidth]{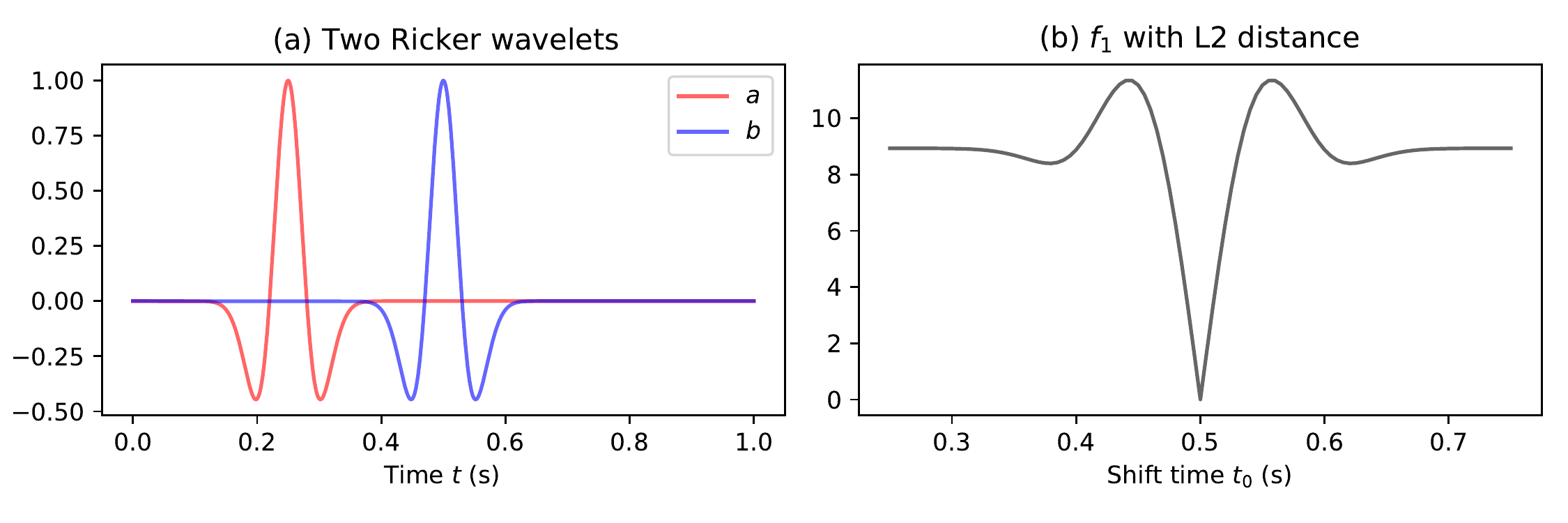}
  \caption{(a): Ricker wavelets $a$ and $b$. (b): The objective function $f_1(t_0)$ with L2 distance.}
  \label{fig:ch4_section2_1}
\end{figure}

First, we investigate the behavior for the time-shift of Ricker wavelets with L2 distance, UOT distance, and mixed L1/Wasserstein distance.
Let $\sigma = 0.03$, $t_0 \in [0.25, 0.75]$, the sampling frequency is $1000$ Hz.
Let $b$ be fixed with the center at $0.5$ s, and $a$ is shifting from left to right, denote $a$ and $b$ as
\begin{align}
  a(t_0,t) &= \left(1 - \frac{(t-t_0)^2}{0.03^2}\right) e^{\frac{-(t-t_0)^2}{2\times 0.03^2}}, \\
  b(t) &= \left(1 - \frac{(t-0.5)^2}{0.03^2}\right) e^{\frac{-(t-0.5)^2}{2\times 0.03^2}},
\end{align}
as in Figure \ref{fig:ch4_section2_1} (a).
We fix $b$ as the reference signal and shift the center of $a$ from $0.25$ s to $0.75$ s.
Define the objective function as
\begin{align}
  f_1(t_0, k) = d(h(a(t_0,t), k), h(b(t), k)),
\end{align}
where $d$ can be UOT distance and mixed L1/Wasserstein distance as equation \eqref{eq:UOT_distance} and \eqref{eq:mixed_distance}.
No normalization method is implied for the L2 distance.
The normalization function $h(\cdot, \cdot)$ can be linear and the exponential normalizations defined by equation \eqref{eq:linear_normalization} and \eqref{eq:exp_normalization}.

To evaluate the UOT distance and mixed L1/Wasserstein distance, we set the entropy regularization parameter $\varepsilon = 1\times 10^{-3}$ to guarantee that the optimal transport distance is evaluated accurately.
We set $\varepsilon_u = 1$ in the UOT distance and set $\lambda_m = 1\times 10^{-10}$ such that both UOT distance and mixed L1/Wasserstein distance have notable results for the time-shift.

The objective function $f_1(t_0)$ with L2 distance is shown in Figure \ref{fig:ch4_section2_1} (b).
One global minimum and two local minima are observed, which is a sign of the cycle-skipping artifact.

\begin{figure}[h]
  \centering
  \includegraphics[width=1\textwidth]{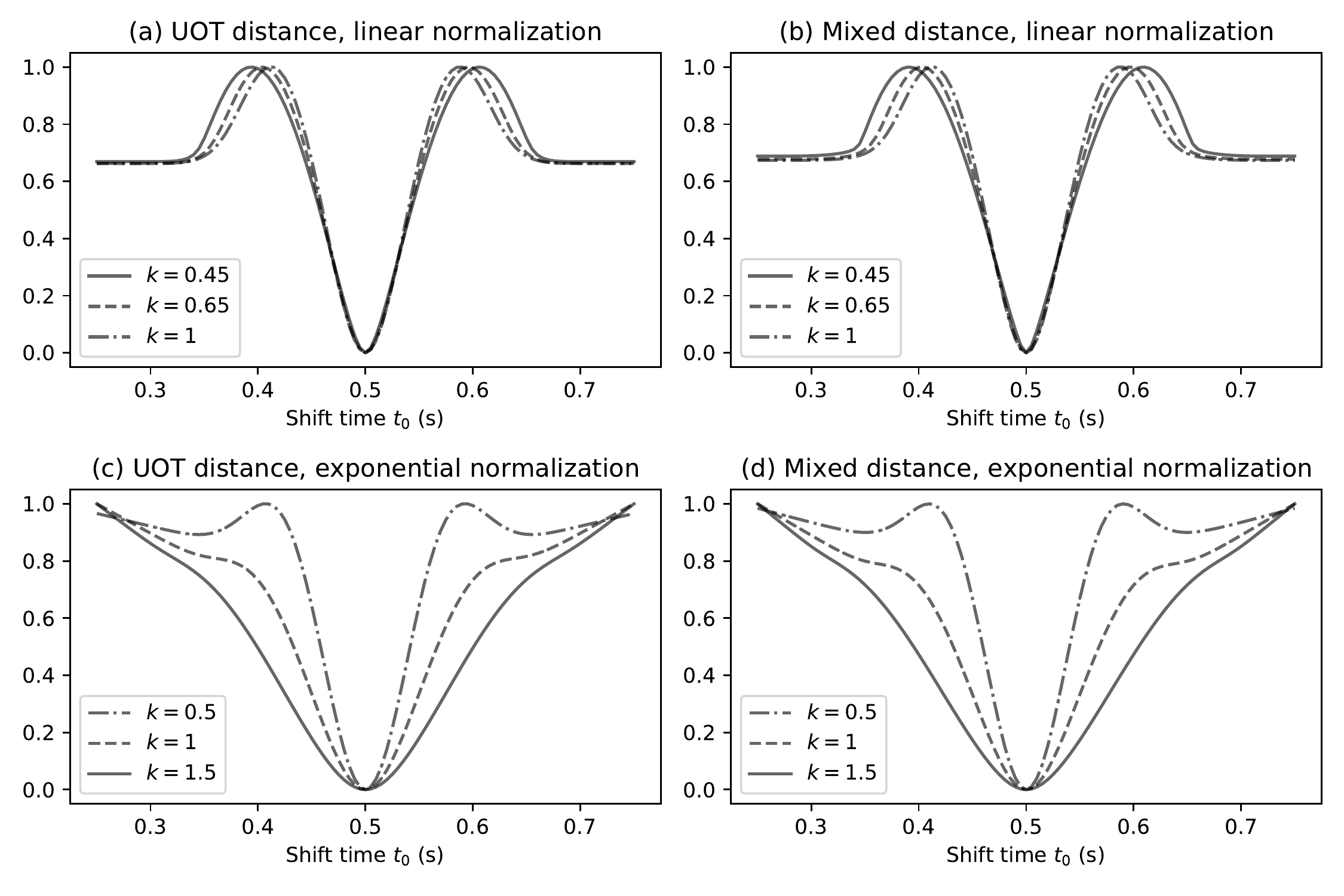}
  \caption{The normalized objective function $f_1(t_0, k)$ with UOT distance, mixed L1/Wasserstein distance and linear normalization, exponential normalization.}
  \label{fig:ch4_section2_2}
\end{figure}

The numerical results of normalized objective function $f_1(t_0, k)$ of both optimal transport based distances with both linear and exponential normalization are shown in Figure \ref{fig:ch4_section2_2}.
Comparing subfigures (a), (b) with subfigures (c) (d), the shape of normalized objective functions are similar for both normalization methods.
Compared to L2 distance, the cycle-skipping artifact is slightly reduced by both distances with linear normalization as shown in subfigures (a) and (b).
The smaller the normalization coefficient $k$ is used, the better performance can be achieved.
However, $k$ can not be less than the absolute value of the minimal value of $a$ and $b$, that is approximately $0.446259$ in this example.
In subfigures (c) and (d), as $k = 0.5$, the normalized objective function is similar to the case of (a) and (b), i.e., with one global minimum and two local minima.
Only one global minimum is obtained with the case $k = 1$ and $k = 1.5$, which means no cycle-skipping artifact occurs in this case.
Compared to L2 distance, both UOT distance and mixed L1/Wasserstein distance can mitigate the cycle-skipping artifact with proper normalization coefficient $k$.
Also, compared to the previous work in \cite{metivier2016optimal,yong2019misfit}, the UOT distance and mixed distance provide more convex behavior than the 1-Wasserstein distance with respect to the time-shift.

\begin{figure}[h]
  \centering
  \includegraphics[width=1\textwidth]{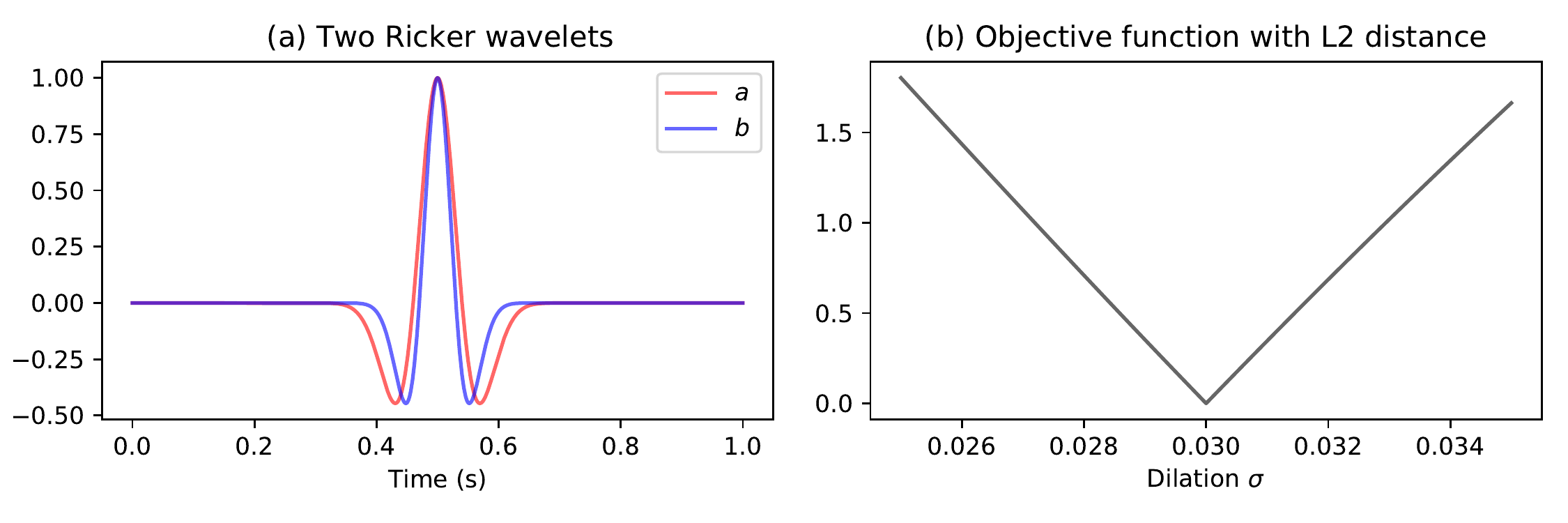}
  \caption{(a): Two Ricker wavelets $a$ and $b$. (b): The objective function $f_2(\sigma_0)$ with L2 distance.}
  \label{fig:ch4_section2_3}
\end{figure}

In the following example, we investigate the behavior with respect to the dilation of the Ricker wavelet.
Fix $t_0 = 0.5$, let $\sigma_0 \in [0.02, 0.04]$, the sampling frequency is still $1000$ Hz.
Let $b$ be fixed with $\sigma = 0.03$, and a is dilating with the change of the $\sigma_0$, denote $a$ and $b$ as
\begin{align}
  a(\sigma_0,t) = \left(1 - \frac{(t-0.5)^2}{\sigma_0^2}\right) e^{\frac{-(t-0.5)^2}{2\sigma_0^2}}, \\
  b(t) = \left(1 - \frac{(t-0.5)^2}{0.03^2}\right) e^{\frac{-(t-0.5)^2}{2\times 0.03^2}}.
\end{align}
One example is shown in Figure \ref{fig:ch4_section2_3} (a).
Define the objective function as
\begin{align}
  f_2(\sigma_0, k) = d(h(a(\sigma_0,t),k), h(b(t), k)),
\end{align}
where $d$ can be UOT distance and mixed L1/Wasserstein distance as equation \eqref{eq:UOT_distance} and \eqref{eq:mixed_distance}.
No normalization method is implied for the L2 distance.
The normalization function $h(\cdot, \cdot)$ can be linear and exponential normalization defined by equation \eqref{eq:linear_normalization} and \eqref{eq:exp_normalization}.
The computation coefficients $\varepsilon_u$, $\lambda_m$ are the same as in the previous shift Ricker example.

The normalized objective function $f_2(\sigma_0)$ with L2 distance is shown in Figure \ref{fig:ch4_section2_3} (b).
Only one global minimum is observed and it is located at point $\sigma_0 = 0.03$, and in this case $a(\sigma_0,t) = b(t)$.

\begin{figure}[h]
  \centering
  \includegraphics[width=1\textwidth]{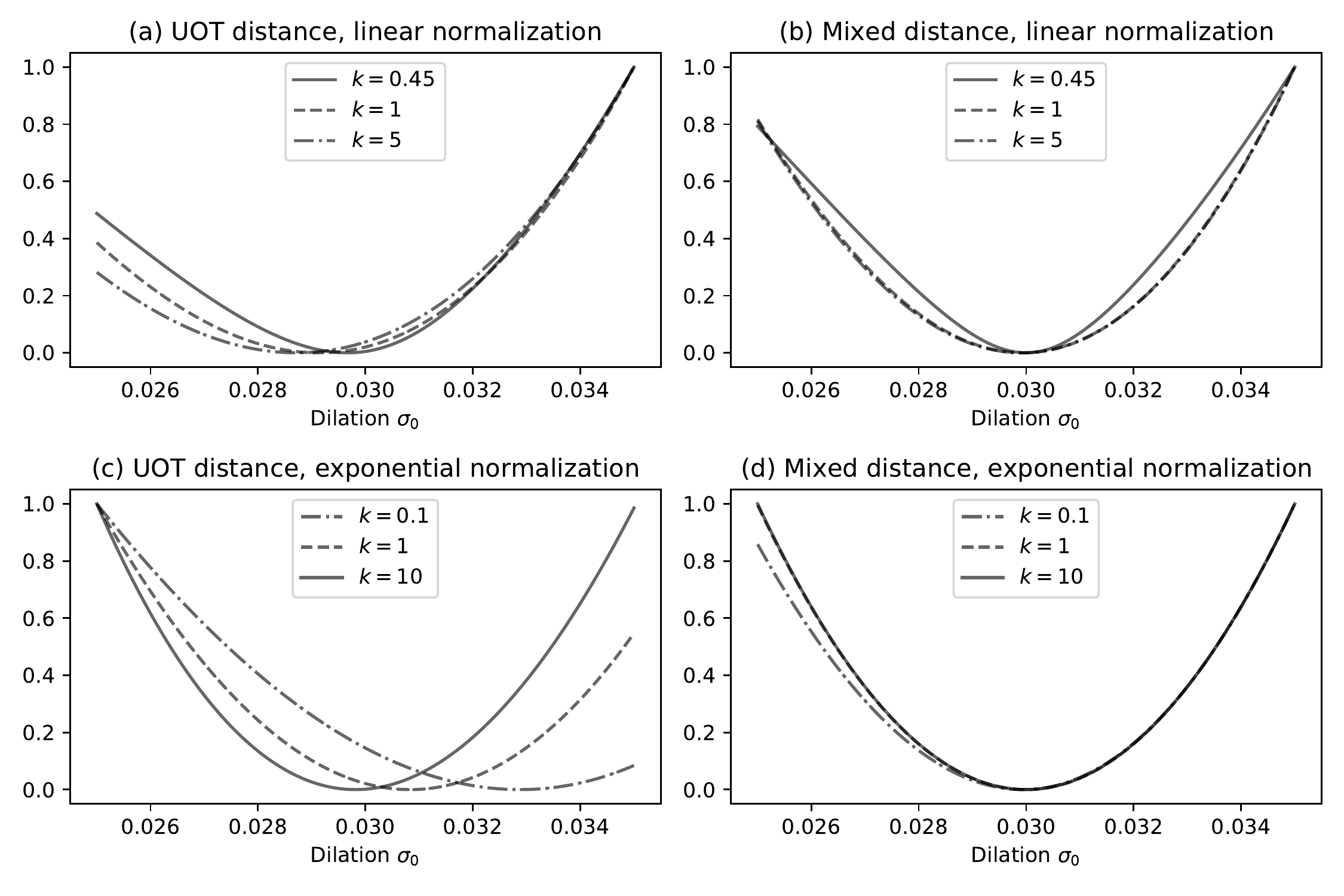}
  \caption{The normalized objective function $f_2(\sigma_0, k)$ with UOT distance, mixed L1/Wasserstein distance and linear normalization, exponential normalization.}
  \label{fig:ch4_section2_4}
\end{figure}

The numerical results of the normalized objective function $f_2(\sigma_0)$ of both optimal transport based distances with both linear and exponential normalization are shown in Figure \ref{fig:ch4_section2_4}.
Subfigure (a) shows the results of UOT distance with linear normalization, only one global minimum can be observed for each $k$.
However, when $k$ is larger, the position of the global minimum tends to be less than $0.03$ which is the global minimum we expect.
Subfigure (c) is the case of UOT distance with exponential normalization.
The position of the global minimum is larger than $0.03$ when $k$ is small.
The position of the global minimum is gradually decreasing when $k$ is increasing, and it will be less than $0.03$ when $k$ is large enough.
The results of mixed L1/Wasserstein distance with both normalizations are shown in subfigures (b) and (d).
There is only one global minimum in each of subfigures (b) and (d), and the global minimum is close to the point $\sigma_0 = 0.03$ for different normalization coefficients $k$.
Compared to the L2 distance, both UOT distance and mixed L1/Wasserstein distance can retain the convex property with respect to $\sigma_0$ with proper normalization and coefficient.

In conclusion, when the linear normalization method is used, the smaller normalization coefficient $k$ leads to a better convex behavior with respect to the time-shift.
However, the $k$ can not be arbitrarily small since it has to be larger than the absolute value of the minimum value of the signals.
Therefore the linear normalization is not encouraged.
On the other hand, both distances with exponential normalization and a larger coefficient $k$ will retain the convex properties with respect to the time-shift and dilation.
Notice that the $k$ can not be arbitrarily large for the UOT distance in order to maintain an accurate convex property for the dilation operation.
Also, to avoid the significant distortion of the waveform, the normalization parameter of exponential normalization should not be too small or large.
In practice, the normalization parameter should be chosen such that the maximum value of normalized signal is approximately in the interval from $1$ and $10$.
In this case, the wavefront of the seismic signal can be slightly amplified and the waveform is not significantly distorted.


\section{Applying the optimal transport based distances in full waveform inversion}

In this section, we formulate the FWI problem with UOT distance and mixed L1/Wasserstein distance by introducing the normalization methods discussed in the previous section.
The wave equation is used as the constraint PDE, and the computation of adjoint sources is provided.

Consider there are $N_s$ sources and $N_r$ receivers in the domain, and let $s = 1, \cdots, N_s$, $r = 1,\cdots, N_r$ be the indexes of the sources and the receivers.
Denote the objective function as
\begin{align}
  J(c,y_1, \cdots, y_{N_s}) = \sum_{s=1}^{N_s} \sum_{r=1}^{N_r} d (h ( Q_r y_s ), h(y_{\mathrm{d},s,r})),
\end{align}
where here $d$ is chosen as one of L2 distance, UOT distance, and mixed L1/Wasserstein distance.
The function $h$ can be linear normalization or exponential normalization, and the normalization coefficient $k$ is omitted.
When the L2 distance is used in the objective function, no normalization method is needed.
The operator $Q_r$ is the recording operator that maps the wavefield generated by the $s$-th source $y_s$ to the signals received by the $r$-th receiver.
The $y_{\mathrm{d},s,r}$ represents the received data by the $s$-th source and the $r$-th receiver.

The constraint PDE is given by
\begin{align}
  \frac{1}{c^2} \frac{\partial^2}{\partial t^2} y_s - \Delta y_s = f_s, \quad s = 1,\cdots, N_s,
\end{align}
where $f_s$ is the function of the $s$-th source.
In practice, a special technique such as absorbing boundary condition (ABC) or perfectly matched layer (PML) is needed to simulate the seismic wave propagating in an unbounded domain.
A numerical PDE method such as finite difference or finite element method is needed to discretize the system and numerically simulate the wave propagation.

Since the PDE is well-posed, it can be written in a compact form as $F_s(c) = y_s$.
Then the reduced objective function is given by
\begin{align} \label{eq:ch4_reduced_equation}
  f(c) = J(c, F_1(c), \cdots, F_{N_s}(c)).
\end{align}
The gradient of $f(c)$ can be achieved through the adjoint state method:
\begin{align} \label{eq:adj_method}
  \nabla f(c) = \sum_{s = 1}^{N_s} \int \frac{-2}{c^3} \left(\frac{\partial^2}{\partial t^2} u_s\right) v_s \ \mathrm{d} t.
\end{align}
Here $v_s$ is the solution of the adjoint equation with $s$-th source
\begin{align}
  \frac{1}{c^2} \frac{\partial^2}{\partial t^2} v_s - \Delta v_s = \tilde f_s,
\end{align}
where $\tilde f_s$ is the adjoint source with respect to the $s$-th constraint equation.
When L2 distance is applied in the objective function, the adjoint source is given by
\begin{align}
  \tilde f_s = - \sum_{r= 1}^{N_r} Q_r^{'} (Q_r y_s - y_{\mathrm{d},s,r}).
\end{align}
When the UOT distance and mixed L1/Wasserstein distance with linear normalization is used in the objective function, the adjoint source is given by
\begin{align}
  \tilde f_s = - \sum_{r= 1}^{N_r} Q_r^{'} \nabla_1 d (h ( Q_r y_s ), h(y_{\mathrm{d},s,r})),
\end{align}
where the $\nabla_1$ is the gradient of $d(\cdot, \cdot)$ with respect to the first term.
When the UOT distance and mixed L1/Wasserstein distance with exponential normalization is used in the objective function, the adjoint source is given by
\begin{align}
  \tilde f_s = - \sum_{r= 1}^{N_r} Q_r^{'} \left(k e^{k Q_r y_s}\right)\nabla_1 d (h ( Q_r y_s ), h(y_{\mathrm{d},s,r})).
\end{align}
Once the gradient $\nabla f(c)$ can be computed, the PDE constrained optimization problem can be solved by the gradient based optimization methods such as conjugate gradient method or L-BFGS method.

\section{Numerical examples}

Three full waveform inversion examples are provided in this section based on the formulation in the previous section.
We compare the numerical results generated by L2 distance, UOT distance, and mixed L1/Wasserstein distance with a two-parameter two-layer example, a cross-well example, and the Marmousi model. 
The exponential normalization method is used for UOT distance and mixed L1/Wasserstein distance.

\subsection{Example 1: Two-parameter two-layer model}

This example shows the difference of objective functions between the L2 distance, UOT distance, and mixed L1/Wasserstein distance of a toy model.
Due to the large size and nonlinear behavior of the FWI problem, we build a simplified two-parameter two-layer velocity model in two-dimension:
\begin{align}
  c(\delta c, z ) = c_0(x,z) + \delta c H(z),
\end{align}
where $H(z)$ is the Heaviside step function along the $z$ direction.
The factor $\delta c$ is the velocity perturbation of the bottom part for the background velocity $c_0(x,z)$.
The background velocity is chosen to be homogeneous with $c_0(x,z) = 1$ km/s.
The model is in a region with $1$ km wide and $1$ km deep, discretized into $101 \times 101$ grid points.
Only one source is used in this example, located at the center of the model and $0.05$ km depth with a $6$ Hz Ricker wavelet.
The sampling frequency is $300$ Hz and the sampling time is $2$ seconds.
There are 11 equally spaced receivers at the top of the region.

Define the objective function as:
\begin{align}
  f_3(\delta c, z) = f(c(\delta c, z)),
\end{align}
where $f(\cdot)$ is defined by equation \eqref{eq:ch4_reduced_equation}.
The true model of this example is $\delta c = 0.05$, $z = 0.51$ which is shown in Figure \ref{fig:ch4_ex2_model}.
We set $\delta c \in [-0.1, 0.2]$ with step size $0.005$, and $z \in [0.4,0.6]$ with step size $0.01$.
Since there is a velocity perturbation between the two layers at the depth $z$, a reflective seismic wave is generated as the seismic wave propagating through the interface, and it will be recorded by the receivers at the top of the model.
For different velocity models $c(\delta c, z)$, the position of the reflector $z$ controls the arriving time of the reflective wave, and the velocity difference $\delta c$ controls the amplitude of the reflective wave.
We generate the received data with the true model $c(0.05,0.51)$.
As the $\delta c$ and $z$ are changing, the reflective waves will interact with the above recorded data, which will cause the cycle-skipping artifact.
We evaluate $f_3$ for each $(\delta c, z)$ by using L2 distance, UOT distance, and mixed L1/Wasserstein distance respectively, the results are shown in Figure \ref{fig:ch4_ex2_result}.
Similar numerical examples with other kinds of optimal transport based distances are provided in the work \cite{engquist2013application, metivier2018optimal, metivier2016optimal}.

\begin{figure}[h!]
  \centering
  \includegraphics[width=0.5\textwidth]{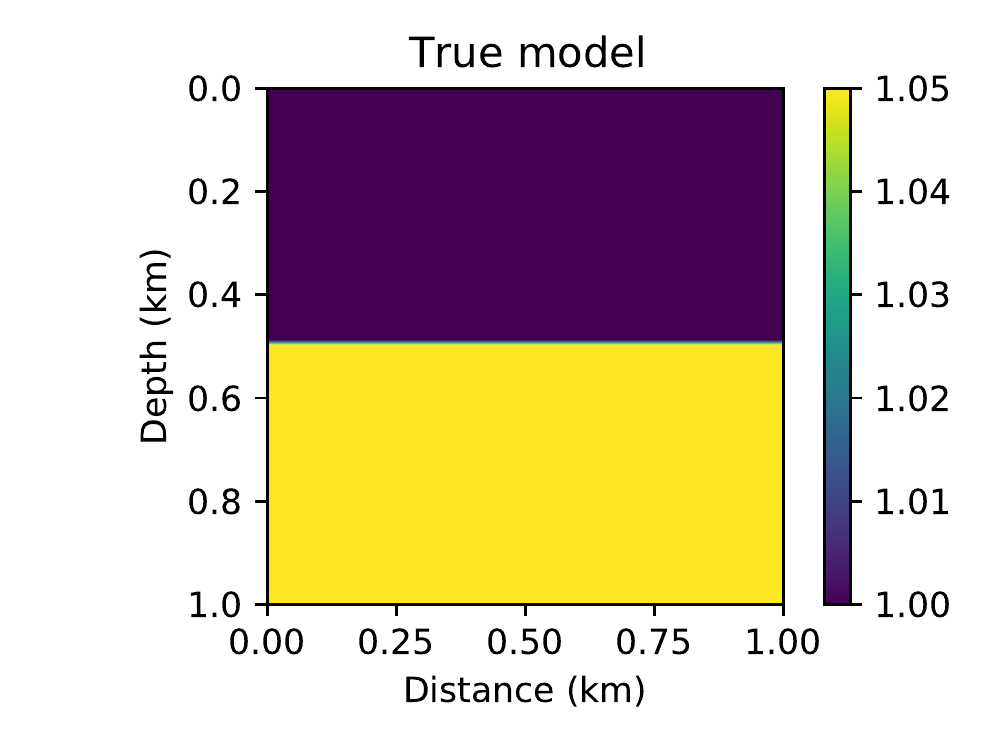}
  \caption{The true velocity model $c(0.05,0.51)$.}
  \label{fig:ch4_ex2_model}
\end{figure}

\begin{figure}[h!]
  \centering
  \includegraphics[width=1\textwidth]{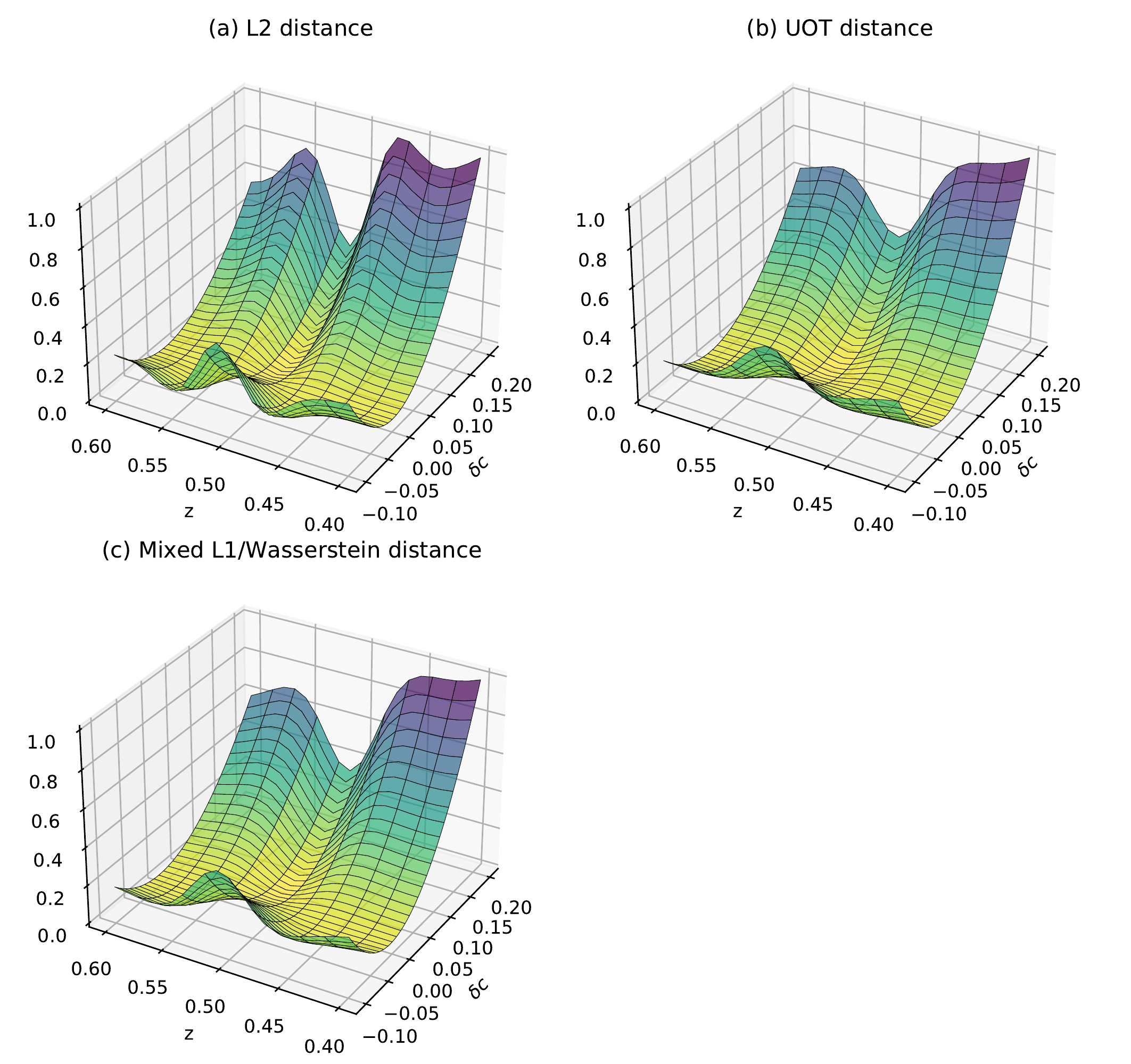}
  \caption{(a), (b), (c): the normalized objective function $f_3(\delta c, z)$ with L2 distance, UOT distance, and mixed L1/Wasserstein distance.}
  \label{fig:ch4_ex2_result}
\end{figure}

In Figure \ref{fig:ch4_ex2_result}, the $z$ axis is the normalized objective function $f_3(\delta c, z)$, and the other two axes are the perturbation $\delta c$ and the position $z$.
The objective function with L2 distance is shown in subfigure (a).
Notice the global minimum is located at the point $(0.05, 0.51)$, and there are several wrinkles in the surface of the objective function around the global minimum.
This suggests that when an initial model that is not close to the global minimum is provided, the optimization algorithm might be trapped in a local minimum due to the wrinkles.

The exponential normalization method is used in this example to compare the difference between signals with UOT distance and mixed L1/Wasserstein distance.
We set the normalization parameter $k = 5 \times 10^4$ such that the maximal value of the normalized signal is approximately in the interval between $1$ and $10$, and the entropy regularization parameter is $\varepsilon = 1 \times 10^{-4}$.
For UOT distance, we set the coefficient of mass balancing term to be $\varepsilon_u = 1$.
And we set the $\lambda_m = 1 \times 10^{-8}$ in the mixed L1/Wasserstein distance.
There are $500$ iterations performed for the computation of UOT and mixed distance.
This numerical example is performed on a server with the cpu model Intel Xeon CPU E7-8891 v4 @ 2.80GHz, and the code is written in the programming language Julia.
There are $12$ workers for the parallel computing are used in this numerical example. 
For each $(\delta c, z)$, $61 * 21 = 1281$ experiments are performed for UOT distance and the mixed distance, and 11 signals are compared for each experiment.
The computation time of the UOT example is 273 seconds, and the computation time of the mixed L1/Wasserstein distance is 228 seconds.

The objective functions with UOT distance and mixed L1/Wasserstein distance are shown in subfigures (b) and (c) respectively.
Compared with subfigure (a), the surface in subfigures (b) and (c) have fewer wrinkle structures.
For an initial model with $\delta c \in [-0.1, 0.2]$ and $z\in [0.4,0.6]$, the optimization algorithm is less likely to be trapped in a local minimum with both UOT distance and mixed L1/Wasserstein distance.
The travel time and shape of the reflected seismic signal is controlled by the parameter $\delta c$ and $z$.
The results of this numerical example are consistent with the Ricker wavelet examples in Section 3.

\subsection{Example 2: Cross-well model}

In this subsection, we perform the full waveform inversion in a two-dimensional cross-well model to investigate the behavior of the update step in the optimization algorithm with direct wave.
The inverse result may be very different from the global minimum when the initial model is inaccurate.
This phenomenon can be demonstrated by the Camembert model \cite{gauthier1986two}.
The previous research shows that the $2$-Wasserstein distance provides more accurate update steps compared to the L2 distance \cite{yang2018application}.
We repeat the Camembert model experiment here to show the optimal transport based distances in this work have the same advantage.

\begin{figure}[h!]
  \centering
  \includegraphics[width=1\textwidth]{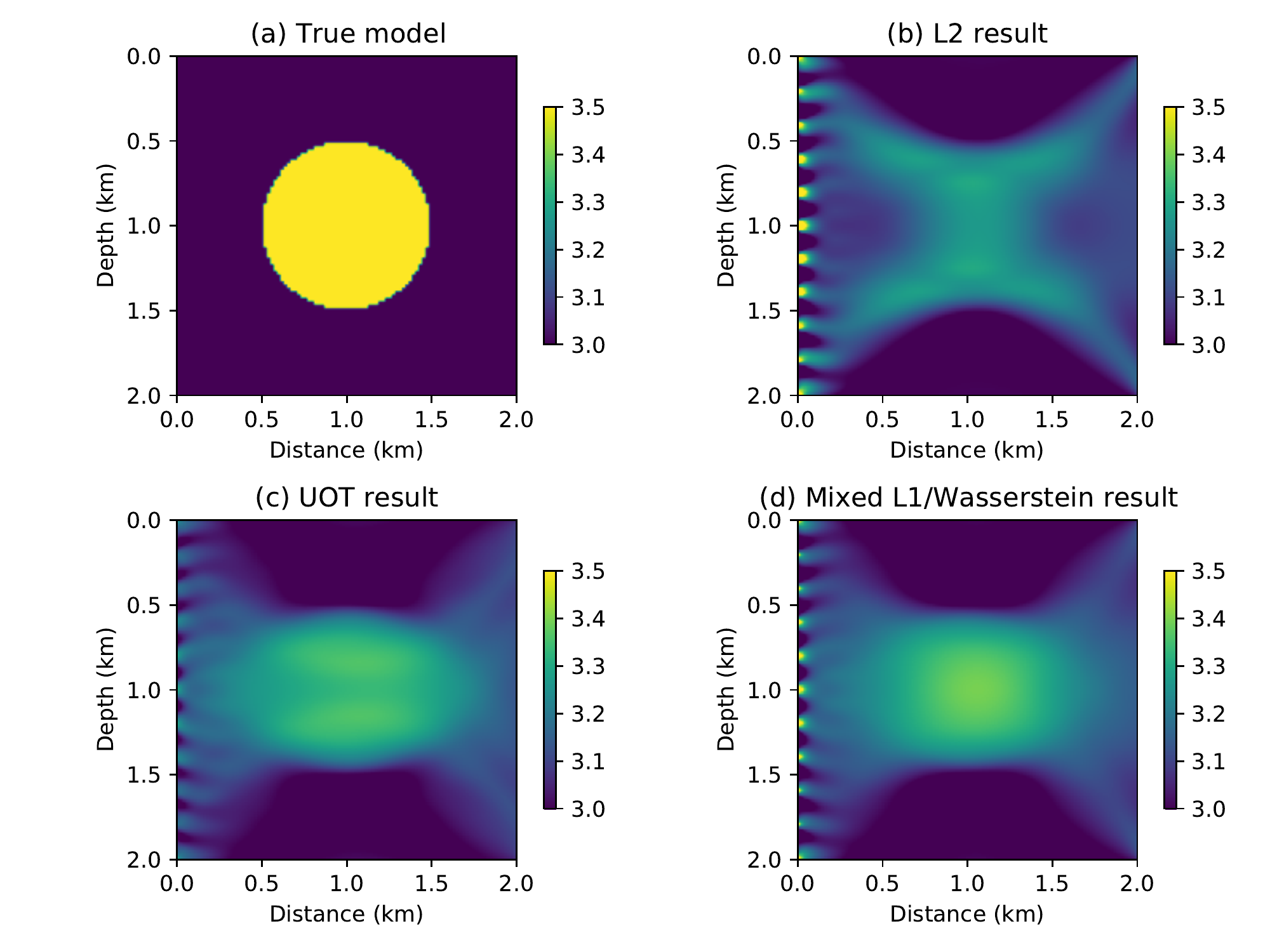}
  \caption{(a): True velocity model. (b): Inverse result with L2 distance. (c): Inverse result with UOT distance and exponential normalization. (d): Inverse result with mixed L1/Wasserstein distance and exponential normalization.}
  \label{fig:ch4_ex3_result}
\end{figure}

The model size is $2$ km by $2$ km, discretized into $101 \times 101$ grids with spatial grid size $0.02$ km.
The true velocity model is given by Figure \ref{fig:ch4_ex3_result} (a).
In the true model, the background velocity is $3$ km/s, and a single circle velocity anomaly is located at the center of the model with radius $0.5$ km and velocity $3.6$ km/s.
There are 11 are equally spaced sources located on the left boundary of the domain, and 101 equally spaced receivers located on the right boundary of the domain.
The synthetic data is generated with $10$ Hz Ricker wavelets and a homogeneous initial velocity model is used with velocity $3$ km/s.

\begin{figure}[h!]
  \centering
  \includegraphics[width=1\textwidth]{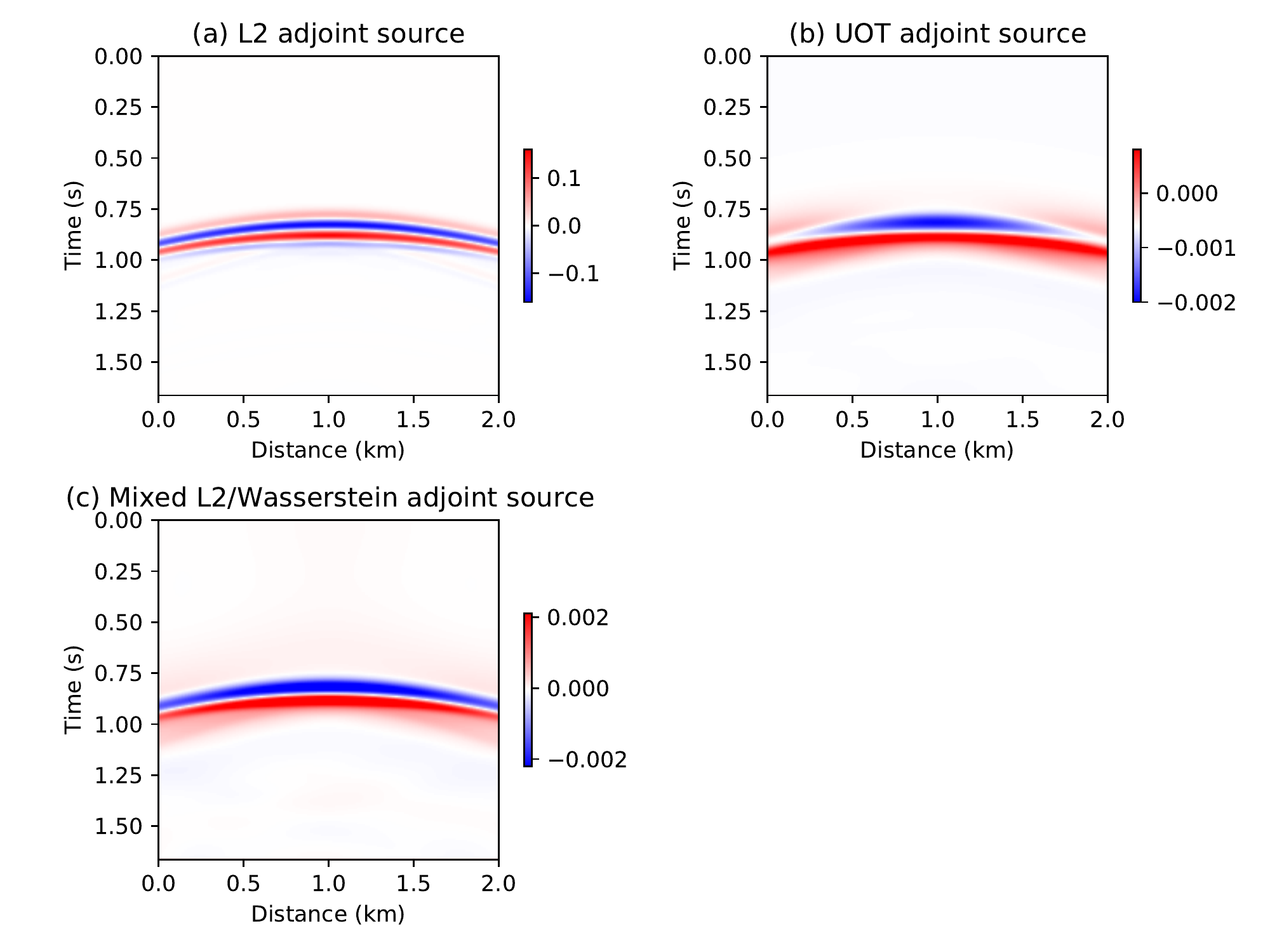}
  \caption{(a), (b), (c): The $6$-th adjoint sources at the first iteration with L2 distance, UOT distance, and mixed L1/Wasserstein distance, the exponential normalization method is used.}
  \label{fig:ch4_ex3_1}
\end{figure}

The inverse results with L2 distance, UOT distance, and mixed L1/Wasserstein distance is compared, the exponential normalization method is used for the optimal transport based distances.
The L-BFGS method with a memory parameter of $5$ is used as the optimization algorithm, and we perform $5$ iterations to show the directions of the velocity model updates.
Figure \ref{fig:ch4_ex3_1} shows the $6$-th adjoint source at the first iteration with different distances.
The adjoint sources generated by UOT distance and mixed L1/Wasserstein distance provide slow transitions on the positions of the seismic wavefront.
The frequency component of the seismic data is lower compared to the L2 case.
This leads to the gradients with fewer large-scale components due to the adjoint state method \eqref{eq:adj_method}.
Also, compared to the trace-by-trace strategy used in \cite{yang2018application}, the adjoint sources in subfigures (b) and (c) are more regular and consistent.

Figure \ref{fig:ch4_ex3_result} (b), (c), (d) display the inverse results with L2 distance, UOT distance, mixed L1/Wasserstein distance respectively.
All three results describe the presence of the velocity anomaly.
However, the L2 result contains abnormal disturbances at the left and right parts of the center, which will provide a wrong velocity update in future iterations.
Compared to the L2 result, both UOT distance and mixed L1/Wasserstein distance provide more regular updates with the shape similar to the velocity anomaly.
This experiment shows that both UOT distance and mixed L1/Wasserstein distance with exponential normalization can reduce the risk of wrong velocity updates, which may cause the optimization algorithm to be trapped in a local minimum.

\subsection{Example 3: Marmousi model}

In this subsection, we compare the inverse results with L2 distance, UOT distance, and mixed L1/Wasserstein distance through a two-dimensional reflection model.
The exponential normalization method is used for the optimal transport based distances.

\begin{figure}[h!]
  \centering
  \includegraphics[width=0.8\textwidth]{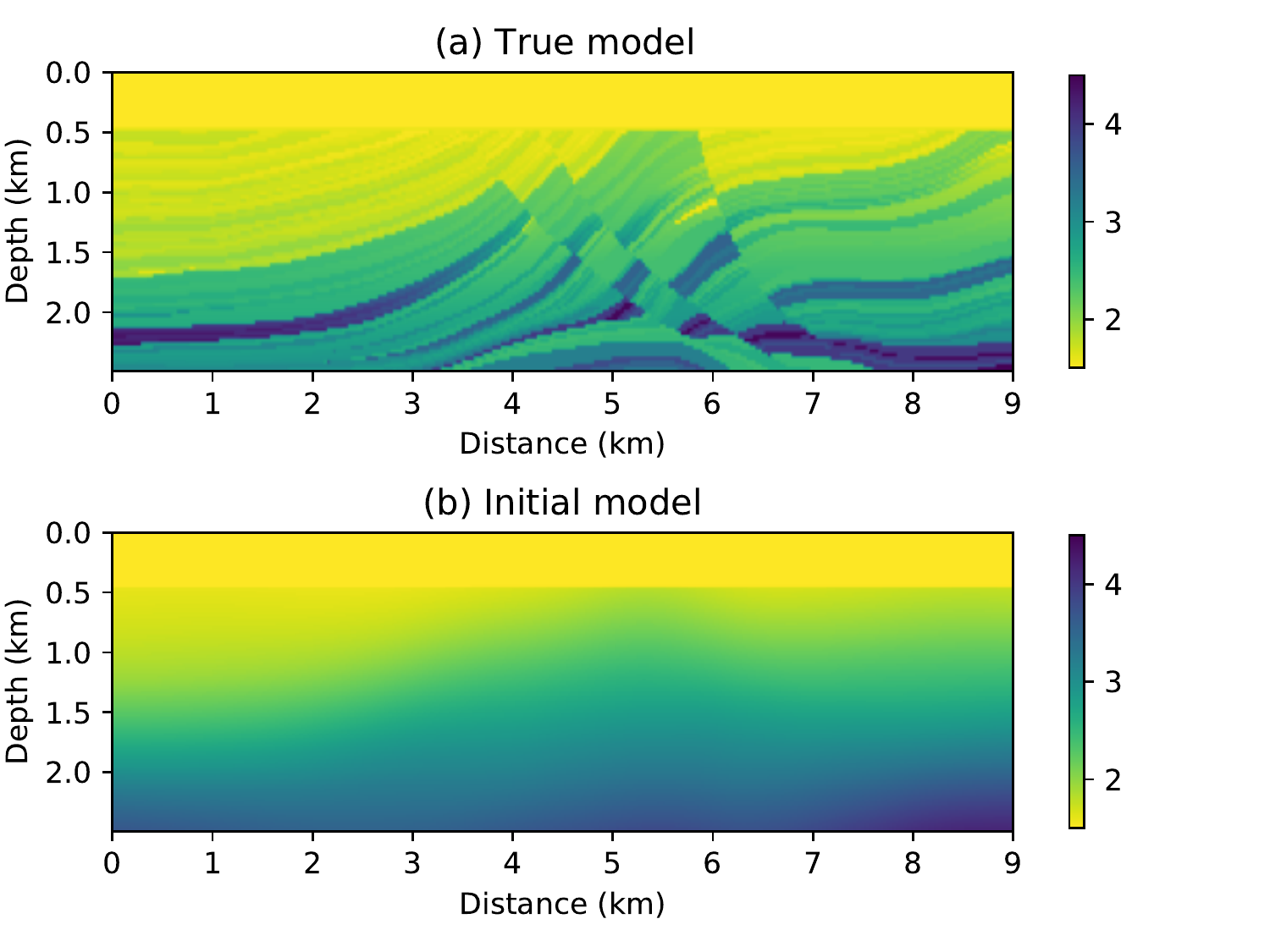}
  \caption{(a): True velocity model. (b): Initial velocity model.}
  \label{fig:ch4_ex4_model}
\end{figure}

As shown in Figure \ref{fig:ch4_ex4_model} (a), the true velocity model is a part of the Marmousi 2 model \cite{martin2006marmousi2} that provides strong velocity differences in both vertical and horizontal directions.
The velocity model is discretized into $84 \times 301 $ grids with the spatial size $0.03$ km.
There are $11$ equally spaced sources and $101$ equally spaced receivers located on the surface of the model.
The initial model is achieved through a two-dimensional Gaussian filter applied to the true model which is strongly smoothened, as shown in Figure \ref{fig:ch4_ex4_model} (b).
The sampling frequency is $400$ Hz, and the recording time is $3$ s.
The synthetic data is generated by the Ricker wavelet with central frequency $5$ Hz as the source function.
The perfectly matched layer technique is performed to simulate the seismic wave propagating in an unbounded domain.

\begin{figure}[h!]
  \centering
  \includegraphics[width=0.8\textwidth]{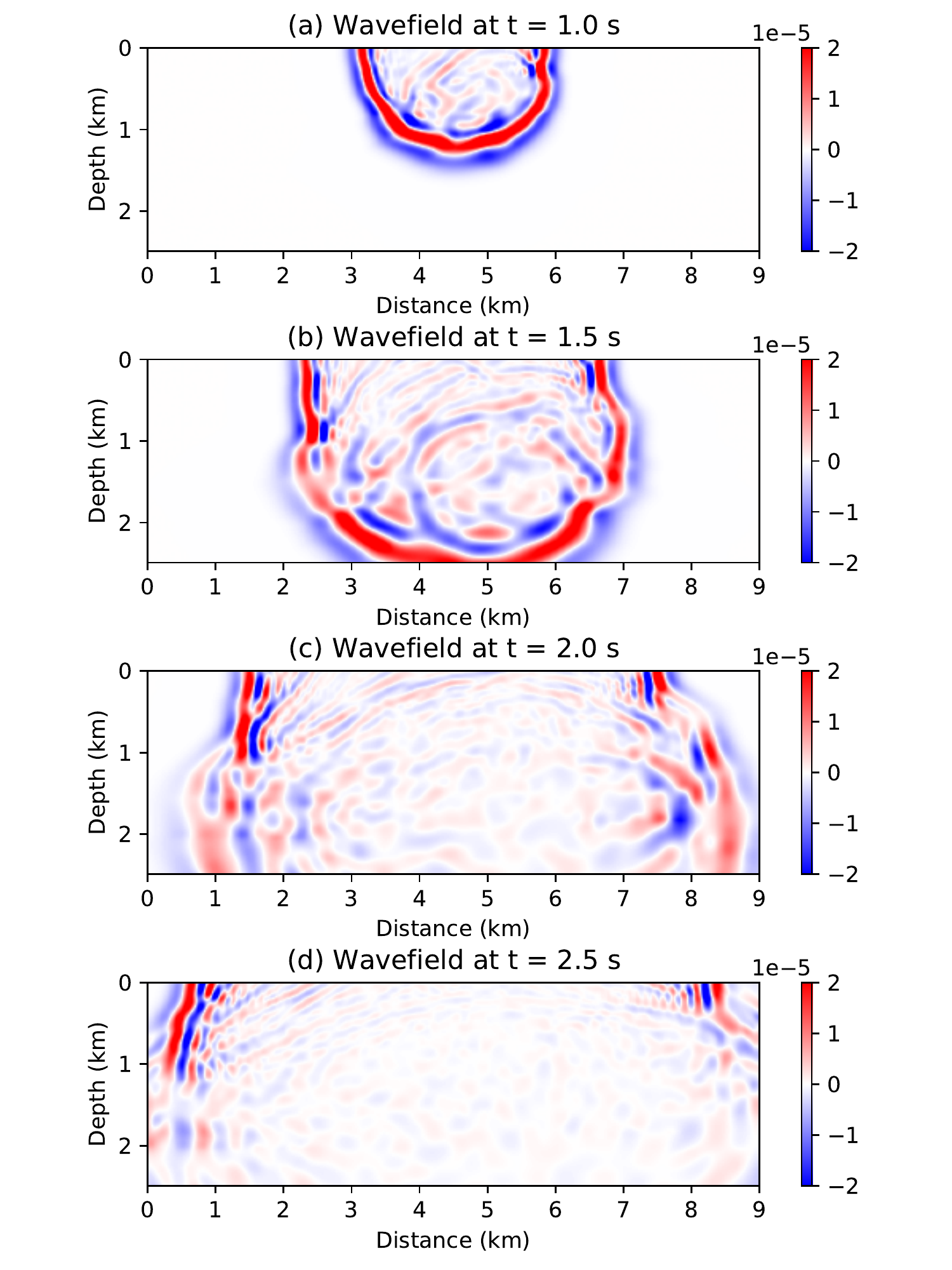}
  \caption{Snapshots of seismic wave generated by the $6$-th source propagating in the domain.}
  \label{fig:ch4_ex4_wavefield}
\end{figure}

\begin{figure}[h!]
  \centering
  \includegraphics[width=1\textwidth]{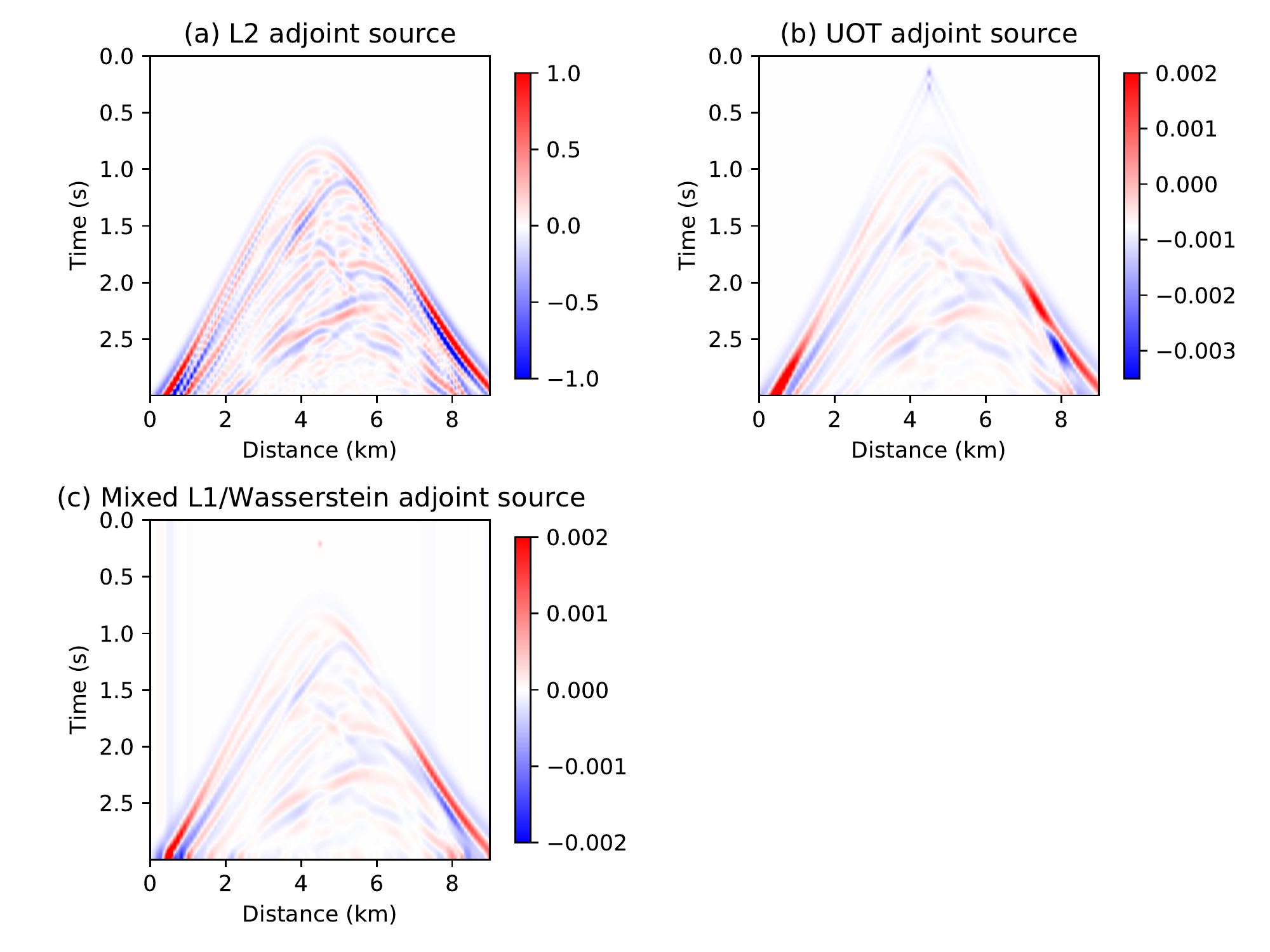}
  \caption{(a), (b), (c): The $6$-th adjoint source at the first iteration with L2 distance, UOT distance, and mixed L1/Wasserstein distance.}
  \label{fig:ch4_ex4_adj}
\end{figure}

\begin{figure}[h!]
  \centering
  \includegraphics[width=0.8\textwidth]{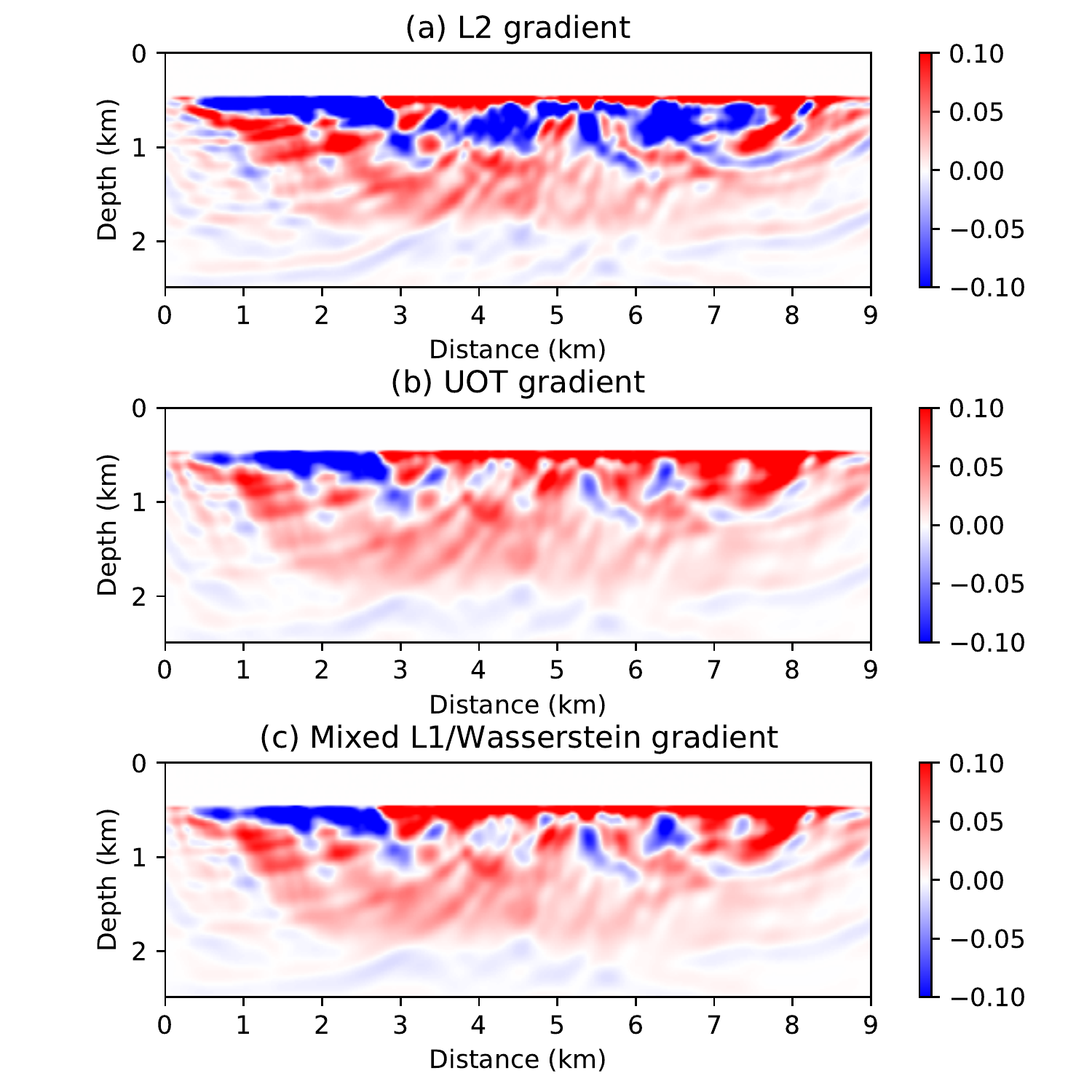}
  \caption{(a), (b), (c): The gradient at the first iteration with L2 distance, UOT distance, and mixed L1/Wasserstein distance.}
  \label{fig:ch4_ex4_gradient}
\end{figure}

Figure \ref{fig:ch4_ex4_wavefield} shows snapshots of the synthetic wavefield, demonstrating the seismic wave propagating in the domain.
Figure \ref{fig:ch4_ex4_adj} shows the adjoint sources of L2 distance, UOT distance, and mixed L1/Wasserstein distance at the first iteration.
Similar to the previous example, the energy of the adjoint sources generated by the optimal transport based distance concentrates on the location of the seismic wavelet, and provides a smoothed waveform of the seismic events.
The first iteration gradients of L2 distance, UOT distance, and mixed L1/Wasserstein distance are shown in Figure \ref{fig:ch4_ex4_gradient}.
Compared to the L2 gradient, the UOT gradient and mixed L1/Wasserstein gradient provide more large-scale structures, which is more sensible on the bottom of the domain.
These large-scale structures will increase the stability of the optimization algorithm.

\begin{figure}[h!]
  \centering
  \includegraphics[width=0.8\textwidth]{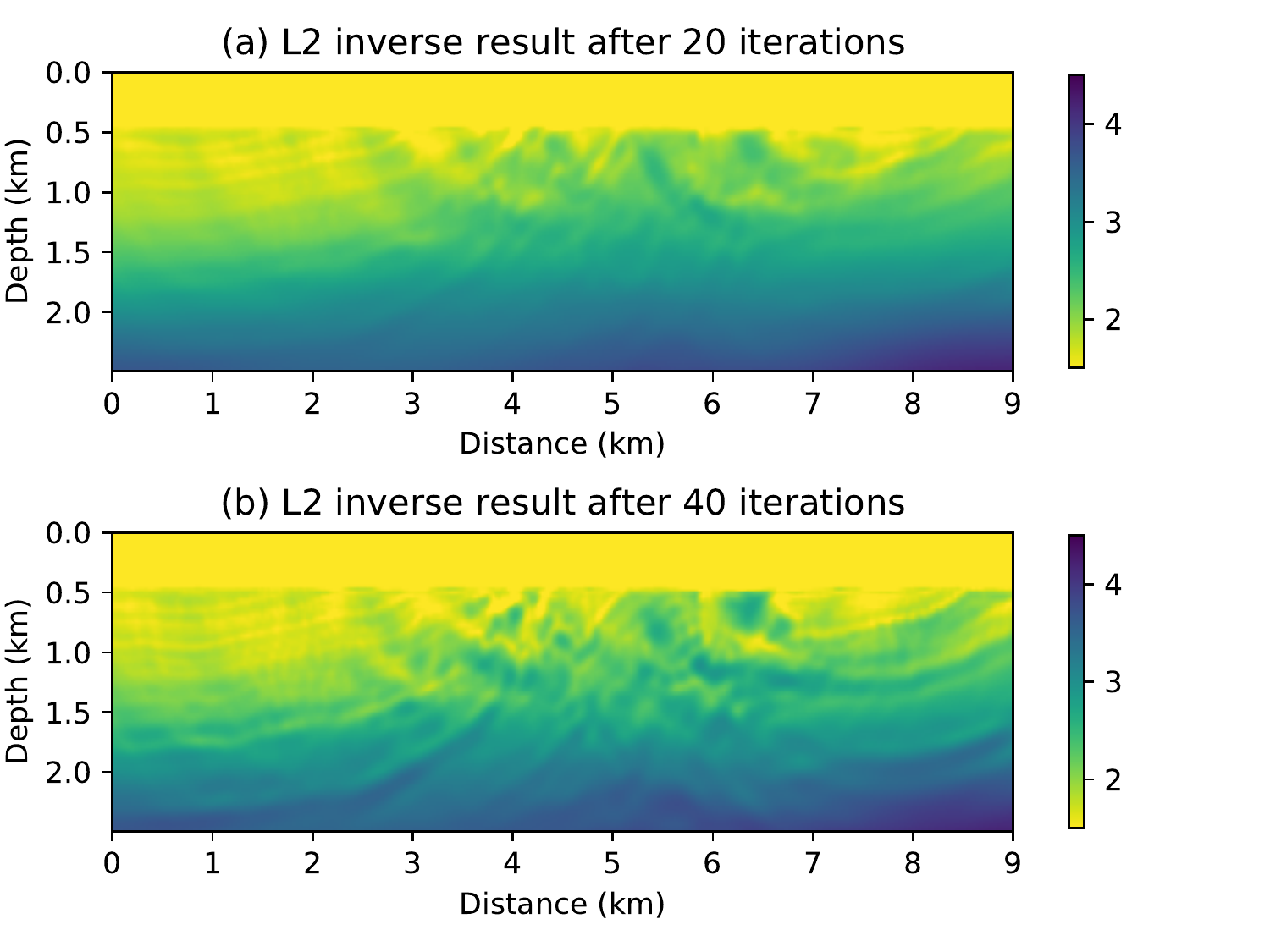}
  \caption{Nonlinear conjugate gradient inverse results with L2 distance after 20 and 40 iterations.}
  \label{fig:ch4_ex4_l2}
\end{figure}

The nonlinear conjugate gradient (NCG) method is performed to minimize the objective function with L2 distance, UOT distance, and mixed L1/Wasserstein distance.
The inverse results with L2 distance are shown in Figure \ref{fig:ch4_ex4_l2}.
The inverse result after 20 iterations and 40 iterations are shown in subfigures (a) and (b) separately.
Compared to the true velocity model, there is a velocity anomaly that exists at near depth $0.75$ km, distance $6.25$ km.
This can be explained as the cycle-skipping artifact since the velocity distribution in the initial model at this area is inaccurate compared to the true velocity model.
In this case, the L2 distance inversion failed to recover the velocity structure of the domain.

\begin{figure}[h!]
  \centering
  \includegraphics[width=0.8\textwidth]{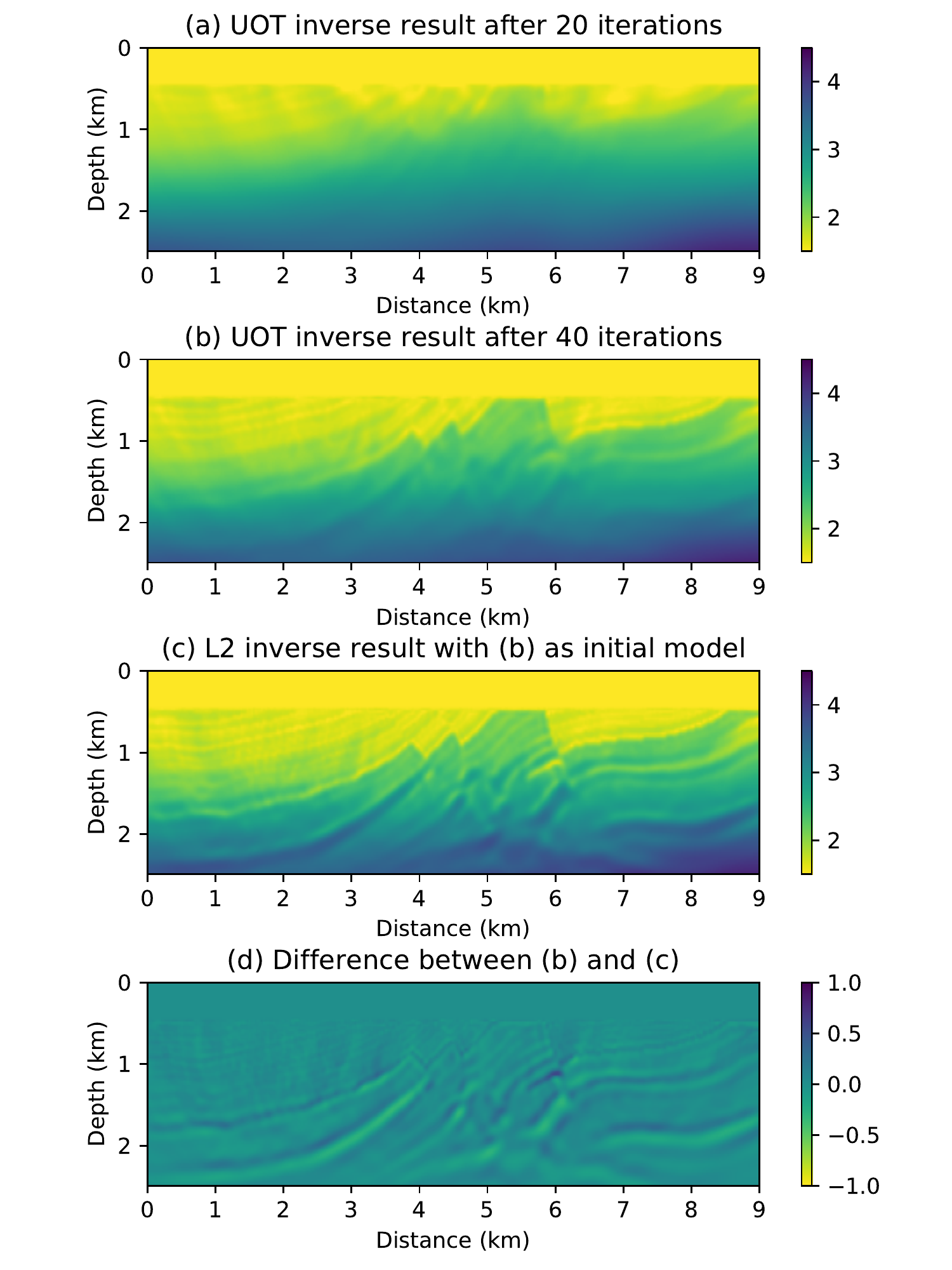}
  \caption{(a), (b): Nonlinear conjugate gradient inverse results with UOT distance after 20 and 40 iterations. (c): Nonlinear conjugate gradient inverse result with L2 distance and (b) as the initial model after 80 iterations. (d): The difference between (b) and (c).}
  \label{fig:ch4_ex4_uot}
\end{figure}

\begin{figure}[h!]
  \centering
  \includegraphics[width=0.8\textwidth]{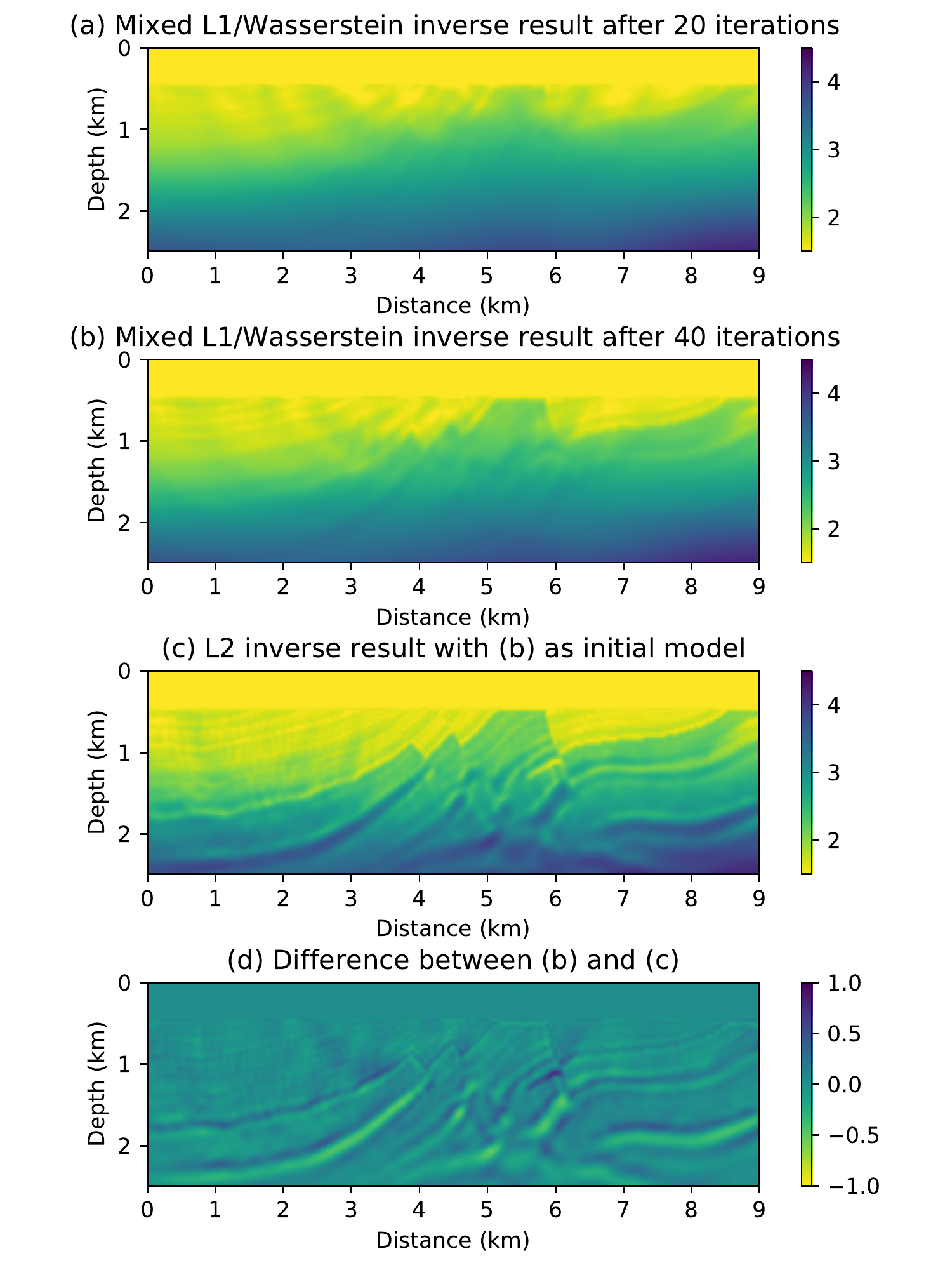}
  \caption{(a), (b): Nonlinear conjugate gradient inverse results with mixed L1/Wasserstein distance after 20 and 40 iterations. (c): Nonlinear conjugate gradient inverse result with L2 distance and (b) as the initial model after 80 iterations. (d): The difference between (b) and (c).}
  \label{fig:ch4_ex4_mixed}
\end{figure}

The inverse results of UOT distance and mixed L1/Wasserstein distance are provided in Figure \ref{fig:ch4_ex4_uot} and \ref{fig:ch4_ex4_mixed} (a) and (b).
Both optimal transport based distances recovered the structure of the true velocity model after $40$ iterations.
The evaluation of UOT distance and mixed L1/Wasserstein distance is much more expensive than the conventional L2 distance.
After the large-scale structure is accurately revealed, the optimal transport based distances can be replaced by the L2 distance in order to achieve the inverse result more efficiently.
With the inverse results of Figure \ref{fig:ch4_ex4_uot} (b) and Figure \ref{fig:ch4_ex4_mixed} (b) as the initial model, $80$ nonlinear conjugate gradient iterations are performed with the L2 distance.
The final inverse results are shown in Figure \ref{fig:ch4_ex4_uot} (c) and Figure \ref{fig:ch4_ex4_mixed} (c) with more detailed velocity structures are revealed.
The difference before and after the additional L2 nonlinear conjugate gradient iterations are shown in Figure \ref{fig:ch4_ex4_uot} (d) and Figure \ref{fig:ch4_ex4_mixed} (d).

\section{Conclusions}

The main contribution of this work is providing a different approach to introduce the optimal transport based distances to the full waveform inversion problem.
To overcome the mass equality restriction, an unbalanced optimal transport distance is introduced.
Furthermore, a mixed L1/Wasserstein distance is constructed, and the convex properties of the proposed mixed distance with respect to shift, dilation, and amplitude change are proved.
To evaluate the above optimal transport based distances, the entropy regularization and Sinkhorn algorithm are used.
Linear and exponential normalization methods that transform the signed signals into positive functions are discussed with numerical examples.
The FWI problem is formulated with both UOT distance and mixed L1/Wasserstein distance, and the computation methods of the adjoint sources are provided.


The numerical examples show that UOT distance and mixed L1/Wasserstein distance with exponential normalization can mitigate the cycle-skipping artifact efficiently compared to the conventional L2 distance.
With a poor initial model, the optimal transport based distances can provide a more accurate update step and increase the stability of the optimization algorithm.
Compared to the optimal transport based distances, L2 distance objective function is sensitive to the initial model but can be evaluated efficiently.
In practice, the inverse problem can be solved by two steps.
First, use UOT distance or mixed L1/Wasserstein distance objective function in the first few iterations to improve the initial model.
Then, use L2 distance objective function in the following iterations to increase the resolution of the inverse result.

There are two reasons that might explain the better performance of the optimal transport based distances.
First, the velocity anomaly between the true velocity model and the initial velocity model changes the arriving time of the seismic event and dilates the shape of the seismic wavelet.
Based on the previous discussion, the optimal transport based distances have more convex behavior compared to the L2 distance with respect to time-shift and dilation.
Another reason is the adjoint sources generated by the optimal transport based distances have less high-frequency components, so the update steps have more large-scale components based on the adjoint state method.
This will decrease the nonlinearity of the optimization problem.

There are different approaches that introducing optimal transport based distance to the FWI problem as we discussed in the introduction.
Compared with some previous approaches, the proposed work has advantages such as providing more convex behavior with respect to the time-shift, providing more regular and consistent adjoint sources.
These are only demonstrated by the numerical experiments.
Sophisticated and rigorous mathematical analysis of different approaches is still expected.

\section{Acknowledgement}


The first author thanks the China Scholarship Council (CSC) for supporting the research. 
The work of the first and second author was funded by NSERC (Natural Science and Engineering Research Council of Canada) through the grant CRDPJ 461179-13 and the Discovery grant RGPIN-2015-06038, RGPIN-2020-04561.
The work of the first and third author was funded by NSERC through the grant CRDPJ 532227-18 and the Discovery grant RGPIN-2019-04830.
The authors would like to thank the industrial sponsors of the Consortium for Research in Elastic Wave Exploration Seismology (CREWES) consortium, and their support through NSERC grant CRDPJ 461179-13.


\bibliographystyle{plain}
\bibliography{references}

\begin{thebibliography}{10}

\bibitem{benamou2003numerical}
Jean-David Benamou.
\newblock {Numerical resolution of an ``unbalanced'' mass transport problem}.
\newblock {\em ESAIM: Mathematical Modelling and Numerical Analysis},
  37(5):851--868, 2003.

\bibitem{benamou2000computational}
Jean-David Benamou and Yann Brenier.
\newblock {A computational fluid mechanics solution to the Monge--Kantorovich
  mass transfer problem}.
\newblock {\em Numerische Mathematik}, 84(3):375--393, 2000.

\bibitem{benamou2001mixed}
Jean-David Benamou and Yann Brenier.
\newblock {Mixed L2-Wasserstein optimal mapping between prescribed density
  functions}.
\newblock {\em Journal of Optimization Theory and Applications},
  111(2):255--271, 2001.

\bibitem{benamou2002monge}
Jean-David Benamou, Yann Brenier, and Kevin Guittet.
\newblock The {Monge}--{Kantorovitch} mass transfer and its computational fluid
  mechanics formulation.
\newblock {\em International Journal for Numerical Methods in Fluids},
  40(1-2):21--30, 2002.

\bibitem{bogachev2007measure}
Vladimir~I. Bogachev.
\newblock {\em {Measure theory}}, volume~1.
\newblock Springer Science \& Business Media, 2007.

\bibitem{chizat2018scaling}
Lenaic Chizat, Gabriel Peyr{\'e}, Bernhard Schmitzer, and Fran{\c{c}}ois-Xavier
  Vialard.
\newblock Scaling algorithms for unbalanced optimal transport problems.
\newblock {\em Mathematics of Computation}, 87(314):2563--2609, 2018.

\bibitem{chizat2018unbalanced}
Lenaic Chizat, Gabriel Peyr{\'e}, Bernhard Schmitzer, and Fran{\c{c}}ois-Xavier
  Vialard.
\newblock Unbalanced optimal transport: Dynamic and kantorovich formulations.
\newblock {\em Journal of Functional Analysis}, 274(11):3090--3123, 2018.

\bibitem{cuturi2013sinkhorn}
Marco Cuturi.
\newblock Sinkhorn distances: Lightspeed computation of optimal transport.
\newblock In {\em Advances in Neural Information Processing Systems}, pages
  2292--2300, 2013.

\bibitem{cuturi2014fast}
Marco Cuturi and Arnaud Doucet.
\newblock {Fast computation of Wasserstein barycenters}.
\newblock 2014.

\bibitem{engquist2013application}
Bjorn Engquist and Brittany~D. Froese.
\newblock {Application of the Wasserstein metric to seismic signals}.
\newblock {\em arXiv preprint arXiv:1311.4581}, 2013.

\bibitem{engquist2016optimal}
Bjorn Engquist, Brittany~D. Froese, and Yunan Yang.
\newblock Optimal transport for seismic full waveform inversion.
\newblock {\em arXiv preprint arXiv:1602.01540}, 2016.

\bibitem{engquist2018seismic}
Bj{\"o}rn Engquist and Yunan Yang.
\newblock Seismic imaging and optimal transport.
\newblock {\em arXiv preprint arXiv:1808.04801}, 2018.

\bibitem{fichtner2010full}
Andreas Fichtner.
\newblock {\em Full seismic waveform modelling and inversion}.
\newblock Springer Science \& Business Media, 2010.

\bibitem{froese2011convergent}
Brittany~D. Froese and Adam~M. Oberman.
\newblock {Convergent finite difference solvers for viscosity solutions of the
  elliptic Monge--Amp{\`e}re equation in dimensions two and higher}.
\newblock {\em SIAM Journal on Numerical Analysis}, 49(4):1692--1714, 2011.

\bibitem{gauthier1986two}
Odile Gauthier, Jean Virieux, and Albert Tarantola.
\newblock Two-dimensional nonlinear inversion of seismic waveforms: numerical
  results.
\newblock {\em Geophysics}, 51(7):1387--1403, 1986.

\bibitem{kantorovich2006translocation}
Leonid~V Kantorovich.
\newblock {On the translocation of masses}.
\newblock {\em Journal of Mathematical Sciences}, 133(4):1381--1382, 2006.

\bibitem{lailly1983seismic}
P.~Lailly.
\newblock {The seismic inverse problem as a sequence of before stack migration:
  Proc. Conf. on Inverse Scattering, Theory and Applications}.
\newblock {\em Expanded Abstracts, Philadelphia, SIAM}, 1983.

\bibitem{li2021novel}
Da~Li.
\newblock Novel optimization schemes for full waveform inversion: Optimal
  transport and inexact gradient projection.
\newblock 2021.

\bibitem{loeper2005numerical}
Gr{\'e}goire Loeper and Francesca Rapetti.
\newblock {Numerical solution of the Monge--Amp{\`e}re equation by a Newton's
  algorithm}.
\newblock {\em Comptes Rendus Mathematique}, 340(4):319--324, 2005.

\bibitem{martin2006marmousi2}
Gary~S. Martin, Robert Wiley, and Kurt~J. Marfurt.
\newblock Marmousi2: An elastic upgrade for marmousi.
\newblock {\em The Leading Edge}, 25(2):156--166, 2006.

\bibitem{metivier2018optimal}
Ludovic M{\'e}tivier, Aude Allain, Romain Brossier, Quentin M{\'e}rigot,
  Edouard Oudet, and Jean Virieux.
\newblock Optimal transport for mitigating cycle skipping in full-waveform
  inversion: A graph-space transform approach.
\newblock {\em Geophysics}, 83(5):R515--R540, 2018.

\bibitem{metivier2016optimal}
Ludovic M{\'e}tivier, Romain Brossier, Quentin Merigot, E.~Oudet, and Jean
  Virieux.
\newblock An optimal transport approach for seismic tomography: Application to
  3d full waveform inversion.
\newblock {\em Inverse Problems}, 32(11):115008, 2016.

\bibitem{metivier2019graph}
Ludovic M{\'e}tivier, Romain Brossier, Quentin Merigot, and Edouard Oudet.
\newblock A graph space optimal transport distance as a generalization of lp
  distances: application to a seismic imaging inverse problem.
\newblock {\em Inverse Problems}, 35(8):085001, 2019.

\bibitem{metivier2016measuring}
Ludovic M{\'e}tivier, Romain Brossier, Quentin M{\'e}rigot, Edouard Oudet, and
  Jean Virieux.
\newblock Measuring the misfit between seismograms using an optimal transport
  distance: Application to full waveform inversion.
\newblock {\em Geophysical Supplements to the Monthly Notices of the Royal
  Astronomical Society}, 205(1):345--377, 2016.

\bibitem{monge1781memoire}
Gaspard Monge.
\newblock M{\'e}moire sur la th{\'e}orie des d{\'e}blais et des remblais.
\newblock {\em Histoire de l'Acad{\'e}mie Royale des Sciences de Paris}, 1781.

\bibitem{pele2009fast}
Ofir Pele and Michael Werman.
\newblock Fast and robust earth mover's distances.
\newblock In {\em 2009 IEEE 12th International Conference on Computer Vision},
  pages 460--467. IEEE, 2009.

\bibitem{piccoli2014generalized}
Benedetto Piccoli and Francesco Rossi.
\newblock {Generalized Wasserstein distance and its application to transport
  equations with source}.
\newblock {\em Archive for Rational Mechanics and Analysis}, 211(1):335--358,
  2014.

\bibitem{plessix2006review}
R.~E. Plessix.
\newblock A review of the adjoint-state method for computing the gradient of a
  functional with geophysical applications.
\newblock {\em Geophysical Journal International}, 167(2):495--503, 2006.

\bibitem{qiu2017mitigating}
Lingyun Qiu, Jaime Ramos-Mart{\'\i}nez, Alejandro Valenciano, Jan Kirkeb{\o},
  and Nizar Chemingui.
\newblock Mitigating the cycle-skipping of full-waveform inversion: An optimal
  transport approach with exponential encoding.
\newblock In {\em SEG 2017 Workshop: Full-waveform Inversion and Beyond,
  Beijing, China, 20-22 November 2017}, pages 1--4. Society of Exploration
  Geophysicists, 2017.

\bibitem{qiu2017full}
Lingyun Qiu, Jaime Ramos-Mart{\'\i}nez, Alejandro Valenciano, Yunan Yang, and
  Bj{\"o}rn Engquist.
\newblock Full-waveform inversion with an exponentially encoded
  optimal-transport norm.
\newblock In {\em SEG Technical Program Expanded Abstracts 2017}, pages
  1286--1290. Society of Exploration Geophysicists, 2017.

\bibitem{santambrogio2015optimal}
Filippo Santambrogio.
\newblock Optimal transport for applied mathematicians.
\newblock {\em Birk{\"a}user, NY}, 55(58-63):94, 2015.

\bibitem{tarantola1984inversion}
Albert Tarantola.
\newblock Inversion of seismic reflection data in the acoustic approximation.
\newblock {\em Geophysics}, 49(8):1259--1266, 1984.

\bibitem{villani2008optimal}
C{\'e}dric Villani.
\newblock {\em Optimal transport: old and new}, volume 338.
\newblock Springer Science \& Business Media, 2008.

\bibitem{virieux2009overview}
Jean Virieux and St{\'e}phane Operto.
\newblock An overview of full-waveform inversion in exploration geophysics.
\newblock {\em Geophysics}, 74(6):WCC1--WCC26, 2009.

\bibitem{yang2018optimal}
Yunan Yang.
\newblock {\em Optimal transport for seismic inverse problems}.
\newblock PhD thesis, 2018.

\bibitem{yang2017analysis}
Yunan Yang and Bj{\"o}rn Engquist.
\newblock {Analysis of optimal transport and related misfit functions in
  full-waveform inversion}.
\newblock {\em Geophysics}, 83(1):A7--A12, 2017.

\bibitem{yang2018application}
Yunan Yang, Bj{\"o}rn Engquist, Junzhe Sun, and Brittany~F. Hamfeldt.
\newblock {Application of optimal transport and the quadratic Wasserstein
  metric to full-waveform inversion}.
\newblock {\em Geophysics}, 83(1):R43--R62, 2018.

\bibitem{yong2019least}
Peng Yong, Jianping Huang, Zhenchun Li, Wenyuan Liao, and Luping Qu.
\newblock {Least-squares reverse time migration via linearized waveform
  inversion using a Wasserstein metric}.
\newblock {\em Geophysics}, 84(5):S411--S423, 2019.

\bibitem{yong2019misfit}
Peng Yong, Wenyuan Liao, Jianping Huang, Zhenchun Li, and Yaoting Lin.
\newblock {Misfit function for full waveform inversion based on the Wasserstein
  metric with dynamic formulation}.
\newblock {\em Journal of Computational Physics}, 399:108911, 2019.

\end{thebibliography}
\end{document}